\documentclass[12pt, reqno]{amsart}
\usepackage{amssymb, amsthm, amsmath, amsfonts}
\usepackage{array, epsfig}
\usepackage{bbm}
\usepackage{graphicx} 

\usepackage{color}

  \evensidemargin
\oddsidemargin
\begin{document}
\newcommand{\Span}{\mathop{\mathrm{span}}\nolimits}
\newcommand{\argmax}{\mathop{\mathrm{argmax}}\nolimits}
\newcommand{\argmin}{\mathop{\mathrm{argmin}}\nolimits}
\newcommand{\Per}{\mathop{\mathrm{Per}}\nolimits}
\newcommand{\Var}{\mathop{\mathrm{Var}}\nolimits}
\newcommand{\Vol}{\mathop{\mathrm{Vol}}\nolimits}
\newcommand{\supp}{\mathop{\mathrm{supp}}\nolimits}
\newcommand{\Cov}{\mathop{\mathrm{Cov}}\nolimits}
\newcommand{\Corr}{\mathop{\mathrm{Corr}}\nolimits}
\newcommand{\sgn}{\mathop{\mathrm{sgn}}\nolimits}
\newcommand{\conv}{\mathop{\mathrm{conv}}\nolimits}
\newcommand{\co}{\mathop{\mathrm{co}}\nolimits}
\newcommand{\law}{\mathop{\mathrm{Law}}\nolimits}
\newcommand{\Intr}{\mathop{\mathrm{Int}}\nolimits}
\newcommand{\diag}{\mathop{\mathrm{diag}}\nolimits}
\newcommand{\dist}{\mathop{\mathrm{dist}}\nolimits}
\newcommand{\Sd}{{\mathbb S}^{d-1}}
\newcommand{\ind}{\mathbbm{1}}

\newtheorem*{def*}{Definition}

\newcommand{\be}{\begin{equation}}
\newcommand{\ee}{\end{equation}}
\newcommand{\bea}{\begin{eqnarray}}
\newcommand{\eea}{\end{eqnarray}}
\newcommand{\beaa}{\begin{eqnarray*}}
\newcommand{\eeaa}{\end{eqnarray*}}

\renewcommand{\proofname}{\bf Proof}
\newtheorem*{rem*}{Remark}
\newtheorem*{cor*}{Corollary}
\newtheorem{cor}{Corollary}
\newtheorem*{conj}{Conjecture}
\newtheorem{prop}{Proposition}
\newtheorem*{prop1}{Proposition 1}
\newtheorem{lem}{Lemma}
\newtheorem*{lem1'}{Lemma $\mathbf{1^\prime}$}
\newtheorem{theo}{Theorem}
\newfont{\zapf}{pzcmi}

\def\R{\mathbb{R}}
\def\Z{\mathbb{Z}}
\def\N{\mathbb{N}}
\def\E{\mathbb{E}}
\def\P{\mathbb{P}}
\def\V{\mathbb{D}}
\def\S{\mathbb{S}}
\def\I{\mathbbm{1}}
\newcommand{\D}{\hbox{\zapf D}}
\newcommand{\Pt}{\widetilde{\mathbb{P}}}
\newcommand{\Et}{\widetilde{\mathbb{E}}}

\newcommand{\bt}{\begin{theo}}
\newcommand{\et}{\end{theo}}
\newcommand{\bl}{\begin{lem}}
\newcommand{\el}{\end{lem}}
\newcommand{\bc}{\begin{cor*}}
\newcommand{\ec}{\end{cor*}}
\newcommand{\br}{\begin{rem*}}
\newcommand{\er}{\end{rem*}}
\newcommand{\bp}{\begin{proof}}
\newcommand{\ep}{\end{proof}}
\newcommand{\bes}{\begin{ex}}
\newcommand{\ees}{\end{ex}}

\title{Convex hulls of multidimensional random walks}

\author{Vladislav Vysotsky}
\address{University of Sussex and St.\ Petersburg Department of Steklov Mathematical Institute}
\email{v.vysotskiy@sussex.ac.uk, vysotsky@pdmi.ras.ru}

\author{Dmitry Zaporozhets}
\address{St.\ Petersburg Department of Steklov Mathematical Institute}
\email{zap1979@gmail.com}

\thanks{This paper was written when V.V. was affiliated to Imperial College London, where his work was supported by People Programme (Marie Curie Actions) of the European Union's Seventh Framework Programme (FP7/2007-2013) under REA grant agreement n$^\circ$[628803].
He was also supported in part by Grant 16-01-00367 by RFBR. The work of D.Z was supported in part by Grant 16-01-00367 by RFBR and by Project SFB 701 of Bielefeld University}

\begin{abstract}
Let $S_k$ be a random walk in $\mathbb R^d$ such that its distribution of increments does not assign mass to hyperplanes. We study the probability $p_n$ that the convex hull $\conv (S_1, \dots, S_n)$ of the first $n$ steps of the walk does not include the origin. By providing an explicit formula, we show that for planar symmetrically distributed random walks, $p_n$ does not depend on the distribution of increments. This extends the well known result by Sparre Andersen (1949) that a one-dimensional random walk satisfying the above continuity and symmetry assumptions stays positive with a distribution-free probability. We also find the asymptotics of $p_n$ as $n \to \infty$ for any planar random walk with zero mean square-integrable increments. 


We further developed our approach from the planar case to study a wide class of geometric characteristics of convex hulls of random walks in any dimension $d \ge 2$. In particular, we give formulas for the expected value of the number of faces, the volume, the surface area, and other intrinsic volumes, including the following multidimensional generalization of the Spitzer--Widom formula (1961) on the perimeter of planar walks: 
$$
\mathbb E V_1 (\conv(0, S_1, \dots, S_n)) = \sum_{k=1}^n \frac{\mathbb E \|S_k\|}{k},
$$
where $V_1$ denotes the first intrinsic volume, which is proportional to the mean width. 

These results have applications to geometry, and in particular, imply the formula by Gao and  Vitale (2001) for the intrinsic volumes of special path-simplexes, called canonical orthoschemes, which are finite-dimensional approximations of the closed convex hull of a Wiener spiral. 
Moreover, there is a direct connection between spherical intrinsic volumes of these simplexes and the probabilities $p_n$.

We also prove similar results for convex hulls of random walk bridges, and more generally, for partial sums of exchangeable random vectors.


\medskip


\end{abstract}

\maketitle

\section{Introduction and results for planar random walks} \label{Sec: Intro}

\subsection{Motivation}
Let $S_n = X_1 + \dots + X_n$ be a random walk in $\R^d$. This paper was motivated by the following question: What is the probability that $\conv(S_1, \dots, S_n)$, the convex hull of the first $n$ steps of the walk, does not include the origin? This is a natural high-dimensional generalization of the classical problem to find the probability that a one-dimensional random walk stays positive (or negative) by time $n$. In this paper we develop a combinatorial approach that answers the question in some particular cases and, importantly, allows one to obtain further results on expected geometric characteristics of the convex hull including its expected number of faces, volume, surface area, and other intrinsic volumes. Most of our main results are presented in the form of {\it exact} non-asymptotic formulas.

Our interest in the probabilities $\P(0 \notin \conv(S_1, \dots, S_n))$ emerged from two different topics. First, we were interested in a multidimensional version of the one-dimensional {\it persistence problem} of finding the probability that a stochastic process (the random walk, in our case) stays above a certain level. Over the past ten years, such problems have drawn a lot of attention from both mathematical and theoretical physics communities; see the survey papers by Aurzada and Simon~\cite{AS} and Bray at al.~\cite{BMS13}. 

Second, we were aware of the direct connection to geometry: for random walks with Gaussian increments, $\frac12 \P(0 \notin \conv(S_1, \dots, S_n))$ equals the $d$-th \emph{spherical intrinsic volume} of a certain path-simplex in $\R^n$ called the {\it canonical orthoscheme}.
This simplex is defined as the convex hull of $n$ vectors whose Gram matrix coincides with the covariance matrix of a standard Brownian motion sampled at times $1, \dots, n$. Spherical intrinsic volumes are spherical analogues of classical Euclidean intrinsic volumes. 
The details on this connection of our problem to geometry are explained below in Section~\ref{1330}, where we also discuss the other geometric properties of canonical orthoschemes.

We were also inspired by two famous results. By Sparre Andersen~\cite[Theorem 2]{Sparre}, for any one-dimensional random walk with continuous symmetric distribution of increments,
\be \label{!!}
\P(S_1 > 0, \dots, S_n > 0) = \frac{(2n-1)!!}{(2n)!!}.
\ee
That is, the probability to stay positive does not depend on the distribution. The other distribution-free result, which is due to Wendel~\cite{Wendel}, also concerns symmetric distributions and describes convex hulls of independent identically distributed random vectors. Let $X_1, \dots, X_n$ be such random vectors in $\R^d$ that satisfy two additional assumptions:
\begin{equation}
\label{cond H_0}
\P(X_1 \in h) =0 \tag{H$_0$}\quad \hbox{for any hyperplane }  h \subset \R^d   \hbox{ passing through the origin},
\end{equation}
and the distribution of $X_1$ is centrally symmetric, i.e.
\be
\label{cond Sym}
X_1\overset{d}{=} -X_1. \tag{S}
\ee
Then
\be \label{conv Wendel}
\P(0 \notin \conv(X_1, \dots, X_n)) = \frac{1}{2^{n-1}}\sum_{k=0}^{d-1}\binom{n-1}{k}.
\ee

Wendel assumed \eqref{cond H_0} to ensure that with probability one, $X_1, \dots, X_n$ are in  {\it
general position}, that is, any $d$ of these vectors are a.s. linearly independent. We will need the stronger assumption
\begin{equation}\label{cond H}
\P(X_1 \in h) =0 \tag{H}\quad \hbox{for any affine hyperplane}\quad h \subset \R^d,
\end{equation}
which in particular guarantees that any one-dimensional projection of $X_1$ has a continuous distribution. We will use this assumption throughout the paper.

\subsection{First results}
There is a similarity between the results of Sparre Andersen and Wendel that stems from the use of combinatorial arguments in their proofs. This motivated our first result, a distribution-free two-dimensional version of \eqref{!!}: 


\begin{theo}\label{main}
Let $d=2$, and assume that \eqref{cond H} and \eqref{cond Sym} hold. Then
\be
\label{symmetric}
\P(0 \notin \conv(S_1,  \dots, S_n)) = \sum_{k=1}^n \frac{(2n-2k-1)!!}{k \cdot (2n-2k)!!}.
\ee
\end{theo}

Let us discuss some corollaries. Here and below we consider the asymptotics as $n \to \infty$. For two positive sequences $a_n$ and $b_n$, the notation $a_n \sim b_n$ means that $\lim_{n \to \infty} a_n/b_n=1$. 

It is not hard to obtain from \eqref{symmetric} (see Section~\ref{prop_proof} below) that
\be \label{log asympt}
\P(0 \notin \conv(S_1, \dots, S_n))  \sim \frac{\log n}{\sqrt{\pi n}}, \qquad d=2.
\ee
Note that this probability is of a higher order of asymptotics than  its one-dimensional counterpart \eqref{!!}, where
\be \label{gamma asympt}
\frac{(2n-1)!!}{(2n)!!} = \frac{\Gamma(n+1/2)}{\Gamma(1/2) \Gamma(n+1)} \sim \frac{1}{\sqrt{\pi n}}.
\ee

Further, since for symmetric random walks one has
$$
\P(0 \notin \conv(S_1, \dots, S_n)) = \P(-S_n \notin \conv(S_1-S_n, \dots, 0)) = \P(S_n \notin \conv(0, S_1, \dots, S_{n-1})),
$$
the expected {\it number of updates} of the convex hull is distribution-free and satisfies
$$
\sum_{k=1}^n \P(S_k \notin \conv(0, S_1, \dots, S_{k-1})) \sim \frac{\sqrt{n} \log n}{2 \sqrt{\pi}}, \qquad d = 2.
$$

The other quantity, which is closely related to the probabilities $\P(0 \notin \conv(S_1, \dots, S_n))$, is the {\it opening solid angle}, denoted by $\Omega_n$, of the convex hull observed from the origin. In the planar case we understand $\Omega_n$ as the arc angle, and so here $\Omega_n = 2 \pi$ if $0$ belongs to the interior of the convex hull and $\Omega_n \le \pi$ if otherwise.

It is easy to see\footnote{Indeed, consider {\it any} set $A \subset \R^d$. If $0 \notin \Intr(\conv(A))$, then by the definition of solid angle, $$\Omega(\conv(A)) := \Bigl| \Bigl \{ \frac{x}{|x|}, x \in \conv(A) \Bigr\} \Bigr| = \frac{1}{2} \int_{\S^{d-1}} \I \bigl( 0 \in \conv(A) | u^\bot \bigr) \sigma(d u),$$ hence $\frac12 - \frac{\Omega(\conv(A))}{|\S^{d-1}|}  = \frac12 \P(0 \notin \conv(A | U^\bot))$. If $0 \in \Intr(\conv(A))$, the l.h.s. in the last equality is negative (it equals $-1/2$) and the r.h.s. equals zero, and thus we have $(\frac12 - \frac{\Omega(\conv(A))}{|\S^{d-1}|} )^+ = \frac12 \P(0 \notin \conv(A | U^\bot))$.} that
\be \label{angle}
\E\,\Bigl(\frac12 - \frac{\Omega_n}{|\S^{d-1}|} \Bigr)^+ = \frac12 \P\bigl(0 \notin \conv(S_1, \dots, S_n) | U^\bot \bigr),
\ee
where: $x^+:=\max(0, x)$ for any real $x$; for any  direction $u \in \S^{d-1}$, the notation $\cdot | u^\bot$ stands for the orthogonal projection onto the hyperplane $u^\bot$ passing through the origin that is orthogonal to $u$; and $U$ is a random vector that is uniformly distributed over the unit sphere $\S^{d-1}$ and independent with the random walk $S_n$. Since for any direction $u$, $\tilde{S}_n:=S_n | u^\bot, n \ge 1,$ is a $(d-1)$-dimensional random walk which satisfies assumptions \eqref{cond H} and \eqref{cond Sym} if the $d$-dimensional walk $S_n$ does so, \eqref{angle} combined with \eqref{!!}  and \eqref{symmetric} imply the distribution-free relations
\be \label{angle 2d}
\E\,(\pi - \Omega_n)^+ =2\pi \frac{(2n-1)!!}{(2n)!!}, \qquad d=2
\ee
and
\be \label{angle 3d} 
\E\,(2\pi -\Omega_n)^+ =2\pi\sum_{k=1}^n \frac{(2n-2k-1)!!}{k \cdot (2n-2k)!!}, \qquad d=3.
\ee

Hence, under the assumptions of Theorem~\ref{main}, the conditional expected discrepancy between the opening angles of the conic hull of $\conv(S_1, \dots, S_n)$ and of a full half-plane containing the hull is also distribution-free and satisfies
$$\E\,\bigl(\pi - \Omega_n \, \bigl | \bigr. \, 0 \notin \conv(S_1, \dots, S_n) \bigr) \sim \frac{ 2\pi}{\log n}, \qquad d=2.$$ 

\medskip

The approach of the present paper does not allow one to generalize \eqref{symmetric} to higher dimensions where it gives non-sharp upper bounds; see \eqref{eq: trivial} and  \eqref{1225} in the next section. Based on numerical simulations for dimensions $d=3$ and $4$, which were further supported by \eqref{angle 3d} in dimension three, we suggested the following hypothesis.
\begin{conj}
Let $d \ge 3$, and assume that \eqref{cond H} and \eqref{cond Sym} hold. Then the probabilities $\P(0 \notin \conv(S_1,  \dots, S_n))$ are distribution-free for any $n \ge 1$.
\end{conj}

Since the time passed from the publication on arXiv of the first version of this paper, this conjecture was fully resolved in our subsequent paper~\cite[Theorem 2.3]{KVZ15} coauthored with Z. Kabluchko. This new work uses an entirely different method, which, however, applies only for perfectly symmetric distributions of increments. 

\subsection{Asymptotic results for general planar random walks} On the contrary, the approach of the present paper allows us, with an additional effort, to obtain an asymptotic version of Theorem~\ref{main} for {\it asymmetric} planar random walks. It also gives asymptotic upper estimates of the probabilities $\P(0 \notin \conv(S_1, \dots, S_n))$ in higher dimensions. Surprisingly, for symmetric walks these  bounds overestimate the true values merely by a constant factor, cf. \eqref{eq: trivial} and \eqref{E RW 0 assymetric} below with~\cite[Theorems 2.3 and 5.1]{KVZ15}.

We now present an asymptotic result for general random walks assuming that the increments have zero mean and a finite covariance matrix $\Sigma$. This matrix must be non-degenerate by assumption~\eqref{cond H}. For such walks,  introduce the following definitions. For any non-zero $u \in \R^d$, denote by $T(u):=\inf \{ k \ge 1:  S_k \notin H(u) \}$ the exit time from the half-space $H(u) := \{z \in \R^d: \langle z, u \rangle \ge 0 \}$, and let
\begin{equation} \label{eq: R}
R(u) := -\frac{\E \langle S_{T(u)}, u \rangle}{\sqrt{\langle \Sigma u, u \rangle }}
\end{equation} 
be the expected normalized distance from $H(u)$ to the exit point $S_{T(u)}$. It is easy to see that this function is positive and angular, that is, independent of $|u|$. We will also show that $R(u)$ is bounded. For random walks with centrally symmetric distribution of increments satisfying~\eqref{cond H}, we have $R(u) \equiv \sqrt2/2$. This becomes clear in the next paragraph if one uses that $T(u)$ is the exit time of the symmetrically distributed random walk $\langle S_n, u \rangle$ from the non-negative half-line.

In dimension one, $R(1)$ and $R(-1)$ are the expected values of ascending and descending, respectively, ladder heights of the random walk normalized by standard deviation of its increments. In this case it is well known (see the proof of Lemma~\ref{lem: half asympt} in the Appendix) that $R(\pm1)$ are finite and 
\begin{equation}
\P(S_1 \ge  0, \dots, S_n \ge 0) \sim \frac{\sqrt{2} R(1)}{\sqrt{\pi n}}, \quad d=1.
\end{equation}
This is actually true even without~\eqref{cond H} while under~\eqref{cond H}, all the inequalities above can be taken to be strict. The  asymptotic order here is the same in the symmetric case (where $R(\pm 1)=\sqrt2/2$, cf. \eqref{!!} and \eqref{gamma asympt}) but of course the probabilities are now distribution-dependent. 

\begin{theo}\label{theo: asympt}
Let $d=2$, and assume that \eqref{cond H} holds and that increments of the random walk $S_n$ have zero mean and a finite covariance matrix $\Sigma$. Then
$$
\P(0 \notin \conv(S_1,  \dots, S_n)) \sim \sqrt{2} \E R \bigl(\Sigma^{-1/2} U \bigr)  \frac{\log n}{\sqrt{\pi n}},
$$
where $U$ is a random vector distributed uniformly over the unit circle $\S^1$ and the expectation above is finite and positive.
\end{theo}

\subsection{Further results and references}
We give a further development to the approach used for our initial problem discussed above. This allows us to obtain new results on a very wide class of {\it geometric characteristics of convex hulls} of general (not necessarily symmetric) multidimensional random walks. In particular, we provide explicit exact formulas for {\it expected intrinsic volumes} of the convex hull. For details we refer the reader directly to Section~\ref{Sec Geom}, with its main result given in Theorem~\ref{second}, and the applications presented in Section~\ref{1324}.

Mean geometric characteristics of convex hulls of planar random walks, for example, the expected number of faces and the expected perimeter, were studied in many papers starting with Spitzer and Widom~\cite{SW} and followed by few other works that include Baxter~\cite{Baxter}, Snyder and Steele~\cite{SS}, and one of the most recent by Wade and Xu~\cite{WX}. It seems that higher-dimensional versions were first considered by Barndorff-Nielsen and Baxter~\cite{Nielsen}, whose work was overlooked by most of the followers, including ourselves. To the best of our knowledge, the probabilities $\P(0 \notin \conv(S_1, \dots, S_n))$ were not considered until the very recent works by Eldan~\cite{Eldan0} and Tikhomirov and Youssef \cite{TY14}, who obtain asymptotic estimates as dimension $d$ increases to infinity for a few special types of random walks.

The paper by Abramson et al.~\cite{Abram} gives an overview and the latest account on the very fine description of the structure of the largest convex minorants of one-dimensional random walks. There are many related papers that consider random walks as the initial step in their studies of convex hulls of continuous time L{\'e}vy processes, and of course there is a huge number of works on convex hulls of Brownian motions. These topics are beyond the scope of our paper. A lot of references can be found in Pitman and Uribe Bravo~\cite{PUB}. There is a survey of results on random convex hulls by Majumdar et al.~\cite{MCR10}.

\subsection{Structure of the paper} This paper is organised as follows. In the next section we present the main tool of our approach, a somewhat technical Proposition~\ref{face_prob}, which immediately implies Theorem~\ref{main} and its analogue for random walk bridges, Theorem~\ref{thm: bridge}. The proposition also serves as the base for the further studies of geometric properties of convex hulls. We also present an asymptotic version of Proposition~\ref{face_prob} for general random walks whose increments have zero mean and finite variance. This result is stated in Proposition~\ref{face_prob asympt} of Section~\ref{1306}, and it readily implies Theorem~\ref{theo: asympt}. 

Section~\ref{Sec Geom} contains Theorem~\ref{second}, our general result on expected geometric characteristics of convex hulls of  multidimensional random walks. This theorem also holds true for random walk bridges and more general, for partial sums of exchangeable random vectors. As the main application of Theorem~\ref{second}, we obtain explicit formulas for expected intrinsic volumes of convex hulls of random walks. In Section~\ref{1330} we consider the special case of random walks with Gaussian increments and present a number of results that explain connections with geometry. 

All the proofs are contained in the last two sections. In Section~\ref{1333} we present our combinatorial results, which are used in Section~\ref{prop_proof} to prove Proposition~\ref{face_prob} and Theorem~\ref{second}. This last section also contains the proof of  Proposition~\ref{face_prob asympt}, which is based on several rather technical statements concerning uniform convergence. The proofs of these statements are moved to the Appendix since they are very different from the purely combinatorial or semi-combinatorial arguments used throughout the paper. However, we think that a {\it uniform} version of {\it Tauberian theorem} (Theorem~\ref{thm: Uniform Tauberian} of the Appendix) should deserve some attention. To our surprise, we did not find a reference to any similar statement.

\section{The main tool for multidimensional random walks}\label{1306}
Denote by
$$
C_n:=\conv(S_0,S_1,\dots, S_n)
$$
the convex hull of the first $n$ steps \emph{including the origin} $S_0:=0$. In the following consideration we will {\it always} refer to $C_n$ as to the convex hull of the random walk. To avoid trivialities, we assume that $n \ge d$; we also recall our convention that~\eqref{cond H} is always satisfied. 

With probability one $C_n$ is a convex polytope with boundary of the form
\be \label{boundary}
\partial C_n=\bigcup_{f\in\mathcal F_n} f,
\ee
where $\mathcal F_n$ is the set of all $(d-1)$-dimensional faces of $C_n$. Almost surely, each face $f$ is a $(d-1)$-dimensional simplex of the form
\begin{equation}\label{a face}
f = \conv(S_{i_1(f)},\dots,S_{i_d(f)})
\end{equation}
for some indices $0\leq i_1(f)<\dots<i_d(f)\leq n$. It is instructive to think that $f$ is obtained by shifting the simplex with vertices $0, S_{i_2(f)}-S_{i_1(f)}, \dots, S_{i_d(f)}-S_{i_1(f)}$ by $S_{i_1(f)}$. We say that the ordered $(d-1)$-tuple $(i_2(f)-i_1(f), \dots, i_d(f)-i_1(f))$ is the {\it temporal structure} of the face and the ordered $d$-tuple $(i_1(f), \dots, i_d(f))$ is the {\it full temporal structure}.

We shall express the probability that $\mathcal F_n$ contains a face of a given temporal  or full temporal structure. In order to stress the combinatorial nature of our result, we prove it in a more general setting for the partial sums $S_k=X_1+\dots+X_k, 1 \le k \le n$, of {\it $n$-exchangeable} increments $X_1, \dots, X_n$. Recalling the definition, this means that for any permutation $\sigma$ of length $n$, $(X_{\sigma(1)}, \dots, X_{\sigma(n)})$ has the same distribution as $(X_1, \dots, X_n)$. We will assume that
\be \label{cond Gen}
\P(S_1, \dots, S_n \hbox{ are in general position})=1 \tag{G}
\ee
to ensure that the faces of $C_n$ still are simplexes with probability one. In other words, any $d$ vectors of $S_1, \dots, S_n$ are linearly independent. This is true, for example, when the exchangeable increments $X_1, \dots, X_n$ have a joint density or when they are independent (so $S_k$ is a random walk) and satisfy \eqref{cond H}. 

Our main example of partial sums with dependent $n$-exchangeable increments is a {\it random bridge} of  length $n$, where we assume by definition that the $n$-th partial sum is a.s. zero. We will be interested in {\it random walk bridges} of two types. For a random walk $S_k$, the {\it difference} bridge is the sequence $S_k - (k/n) S_n, 1 \le k \le n$; and the distribution of the {\it conditional} bridge is given by conditioning on $S_n=0$. We understand the latter as the well-defined limit of the corresponding conditional distributions $\P\bigl(\, \cdot \, \bigl| \bigr. |S_n| \le r \bigr)$  as $r \to 0+$. For example, this limit exists if the distribution of increments of the walk has continuous or bounded density and the density of the distribution of $S_n$ is positive at $0$. It is easy to see that the first $n$ values of a random walk bridge of length $n+1$ of either type satisfy~\eqref{cond Gen} if the underlying random walk satisfies~\eqref{cond H}. 

It turns out that the probability that the convex hull $C_n$ contains a face of a given {\it full} temporal structure is distribution-free for random walk bridges and for random walks  with symmetrically distributed increments. Although this probability is not distribution-free for general walks, the probability that $C_n$ contains a face of a given temporal structure is. More precisely, we have following result.   


\begin{prop}\label{face_prob}
For any $d \ge 1$, let $0 \le i_1 < \dots < i_d \le n$ be any indices.
\begin{enumerate}
\item If the partial sums $S_k$ of $n$-exchangeable random vectors in $\R^d$ satisfy \eqref{cond Gen}, then
\be \label{temporal}
\sum_{i=0}^{n-i_d+i_1} \P(\conv(S_i, S_{i+i_2-i_1},\dots,S_{i+ i_d-i_1}) \in \mathcal F_n) = \frac{2}{(i_2 - i_1) \cdot ... \cdot (i_d-i_{d-1})}.
\ee
Moreover, if $S_1, \ldots, S_n, 0$ is random bridge of length $n+1$, then
\be \label{pinned bridge}
\P(\conv(S_{i_1},\dots,S_{i_d}) \in \mathcal F_n )=\frac{2}{(i_2 - i_1) \cdot ... \cdot (i_d-i_{d-1})(n-i_d+i_1+1)}.
\ee
\item  If $S_k$ is a random walk\footnote{\label{note} As explained in Section~\ref{prop_proof}, \eqref{pinned} is true for the partial sums of $n$-exchangeable random vectors $X_1, \dots, X_n$ if \eqref{cond Gen} holds and all the $2^n$ $n$-tuples $(\pm X_1, \dots, \pm X_n)$ have the same distribution; note that Wendel's result~\eqref{conv Wendel} also holds true under these relaxed assumptions. Here is an example of such distributions: if $d=1$ so are the coordinates of any random vector $X$ in $\R^n$ with a rotationally invariant distribution. In this case $X_1, \dots, X_n$ are not i.i.d. unless $X$ is a multiple of a standard Gaussian vector.} in $\R^d$ and \eqref{cond H} and \eqref{cond Sym} hold, then
\be \label{pinned}
\P(\conv(S_{i_1},\dots,S_{i_d}) \in \mathcal F_n )=2\frac{(2i_1-1)!!}{(2i_1)!!}\frac{(2n-2i_d-1)!!}{(2n-2i_d)!!}\prod_{k=1}^{d-1}\frac{1}{i_{k+1}-i_k},
\ee
where by convention $(-1)!!=1$.
\end{enumerate}

\end{prop}


The first application of this result concerns the expected number of faces of $C_n$ that contain the origin as a vertex. Denote the set of such faces by $\mathcal F_n'$ and note that (under~\eqref{cond Gen}!)
$$
\I(\mathcal F_n' \neq \varnothing) \stackrel{\text{a.s.}}{=} \I (0 \in \partial C_n ) \stackrel{\text{a.s.}}{=} \I(0 \notin \conv(S_1,\dots, S_n) ).
$$
Then \eqref{pinned} immediately implies that for symmetric random walks,
\begin{equation}\label{1225}
\E |\mathcal F_n'|=2\sum_{1 \leq i_2 <\dots<i_d\leq n}\frac{(2n-2i_d-1)!!}{i_2 \cdot (2n-2i_d)!!}\prod_{k=2}^{d-1}\frac{1}{i_{k+1}-i_k}.
\end{equation}
This proves Theorem~\ref{main} since for $d=2$, we have
\begin{equation}
\label{eq: 2 faces in 2D}
|\mathcal F_n'|=\left\{
                 \begin{array}{ll}
                   2, & 0 \notin \conv(S_1, \dots, S_n), \\
                   0, & 0 \in \conv(S_1, \dots, S_n).
                 \end{array}
               \right.
\end{equation}
In higher dimensions, \eqref{1225} gives only an upper bound, as follows by 
\begin{equation} \label{eq: trivial}
\P(0 \notin \conv(S_1,  \dots, S_n) \le \E|\mathcal F_n'| /d.
\end{equation}

By the same reasoning, from \eqref{pinned bridge} we obtain the following version of Theorem~\ref{main} for random walk bridges. 

\begin{theo} 
\label{thm: bridge}
Let $S_1, \ldots, S_{n+1}$ be either the difference bridge or a well-defined conditional bridge (both of length $n+1$) of a random walk in $\R^2$ that satisfies~\eqref{cond H}. Then
$$\P(0 \notin \conv(S_1,  \dots, S_n)) = \sum_{k=1}^n \frac{1}{k(n-k+1)}.$$
\end{theo}
We stress that no additional assumption other than \eqref{cond H} is required.

For the asymptotics, it follows that (see Section~\ref{prop_proof}) for a random walk under \eqref{cond Sym} and~\eqref{cond H},
\be \label{E RW 0}
\E_{RW} |\mathcal F_n'| \sim \frac{2 (\log n)^{d-1}}{\sqrt{\pi n}},
\ee
while for a random walk bridge of length $n+1$ (under \eqref{cond H}),
\be \label{E BR 0}
\E_{Br} |\mathcal F_n'| \sim \frac{2d (\log n)^{d-1}}{n}.
\ee

We conclude this section with an asymptotic version of Part 2 of Proposition~\ref{face_prob} with $i_1=0$ for general (not necessarily symmetric) random walks. Recall that the  function $R(u)$ was defined in~\eqref{eq: R}.

\begin{prop}\label{face_prob asympt}
Let $S_k$ be a random walk in $\R^d, d \ge 2,$ with increments that have zero mean, a finite covariance matrix $\Sigma$, and satisfy  \eqref{cond H}. Let $U$ be a random vector distributed uniformly over the unit sphere $\S^{d-1}$. Then for any sequence $h_n$ tending to infinity such that $h_n =o( n)$, we have
$$
\P(\conv(0, S_{i_2},\dots,S_{i_d}) \in \mathcal F_n ) = \Bigl(2  \sqrt{\frac{2}{\pi}} + o(1) \Bigr) \frac{ \E R(\Sigma^{-1/2} U)}{i_2 \sqrt{n - i_d+1}} \prod_{k=2}^{d-1}\frac{1}{i_{k+1}-i_k}
$$
uniformly in  $1 \le i_2 < \dots < i_d \le n$ such that $\min(i_2, i_3 -i_2, \ldots, i_d - i_{d-1}, n-i_d) \ge h_n$, and the $o(1)$ term is uniformly bounded.
\end{prop}

Similarly to \eqref{E RW 0}, this gives (see Section~\ref{prop_proof})  the asymptotics
\begin{equation}
\label{E RW 0 assymetric}
\E_{RW} |\mathcal F_n'| \sim 2 \sqrt{2} \E R(\Sigma^{-1/2} U) \frac{ (\log n)^{d-1}}{\sqrt{\pi n}}.
\end{equation}
Then Theorem~\ref{theo: asympt} readily follows by~\eqref{eq: 2 faces in 2D}.

\section{Geometric properties of convex hulls in $\R^d$} \label{Sec Geom}
\subsection{Expected additive functionals of faces} 
For a further application of Proposition~\ref{face_prob}, we sum in \eqref{temporal} over all possible indices to obtain that the expected number of faces in the convex hull satisfies
\begin{equation}\label{E faces}
\E |\mathcal F_n|=2\sum_{\substack{j_1+\dots+j_{d-1}\leq n\\ j_1,\dots, j_{d-1}\geq1}}\frac{1}{j_1 \cdot ... \cdot j_{d-1}}.
\end{equation}
Comparing \eqref{E faces} and \eqref{pinned bridge}, we see that
$$\E^{(d)} |\mathcal{F}_n| = \sum_{k \le n} \E^{(d-1)}_{Br} |\mathcal{F}'_{k-1}|,$$
where the upper indices show the dimension, hence by \eqref{E BR 0},
\be \label{E RW}
\E |\mathcal F_n| \sim 2(\log n)^{d-1}.
\ee
We stress that these formulas are valid under \eqref{cond Gen} only, and \eqref{cond Sym} is not required.


For $d=2$, \eqref{E faces} was proved by Baxter~\cite{Baxter}. We first generalized his argument to higher dimensions, but then found a more direct and intuitive proof for Part 1 of Proposition~\ref{face_prob} presented below in Section~\ref{prop_proof}. Later we discovered that such generalization was already done by Barndorff-Nielsen and Baxter~\cite{Nielsen} who extended the proof of~\cite{Baxter}.

We followed the steps of Baxter~\cite{Baxter} and Snyder and Steele~\cite{SS} (both papers considered only the planar case) to obtain the following generalization of \eqref{E faces}. Let $g:\R^{d \times (d-1)}\to\R$ be any  non-negative Borel function. As we noted above, with probability one $C_n$ is a convex polytope with faces of the form~\eqref{a face}, hence we can represent nearly any geometric property of a face $f$ of $C_n$ in terms of $$g\left(S_{i_2(f)}-S_{i_1(f)},\dots,S_{i_d(f)}-S_{i_{d-1}(f)}\right)$$ for some symmetric function $g$. This quantity has the same expectation for all faces with the same temporal structure, and a conditional version of \eqref{temporal} (see \eqref{temporal cond} below) readily yields the following result.


\begin{theo}\label{second}
Let $S_k$ be partial sums of $n$-exchangeable random vectors in $\R^d$, $d \ge 1$. If \eqref{cond Gen} holds, then for
$$
G_n:=\sum_{f\in\mathcal F_n} g\left(S_{i_2(f)}-S_{i_1(f)},\dots,S_{i_d(f)}-S_{i_{d-1}(f)}\right),
$$
we have that
$$
\E G_n=2\sum_{1 \le i_1 < \dots < i_{d-1} \le n}\frac{\E g(S_{i_1}, S_{i_2} - S_{i_1}, \dots, S_{i_{d-1}} - S_{i_{d-2}})}{i_1 (i_2 - i_1) \cdot ... \cdot (i_{d-1} - i_{d-2})}.
$$
\end{theo}
Notice that if $S_n$ is a random walk that satisfies \eqref{cond H}, then the $d-1$ arguments of $g$ in the definition of $G_n$ are independent. In this case $\E G_n$ can be written as
\begin{equation}\label{1454}
\E G_n=2\sum_{\substack{j_1+\dots+j_{d-1}\leq n\\ j_1,\dots, j_{d-1}\geq1}}\frac{\E g\left(S^{(1)}_{j_1}, S^{(2)}_{j_2}, \dots, S^{(d-1)}_{j_{d-1}}\right)}{j_1\cdot ... \cdot j_{d-1}},
\end{equation}
where $S^{(1)}_n, \dots, S^{(d-1)}_n$ are independent copies of the random walk $S_n$.

We proved Theorem~\ref{second} being unaware of the work of Barndorff-Nielsen and Baxter~\cite{Nielsen}, who gave no general statement of this type but did a similar consideration and obtained many of the results discussed in the next section as applications of Theorem~\ref{second}. Our proof uses both combinatorial and probabilistic reasoning and in our opinion, is more transparent than that of~\cite{Nielsen}.

The latter proof rests on the smart combinatorial argument proposed by Baxter~\cite{Baxter}. It is based on the simple fact that none of the $n!$ permutations of the increments $X_1, \dots, X_n$   change the distribution of the partial sums. Similarly, Wendel's proof of \eqref{conv Wendel} uses that all the $2^n$ possible $n$-tuples $(\pm X_1, \dots, \pm X_n)$ have the same distribution. If this holds true, his argument works for any random vectors $X_1, \dots, X_n$ in general position (so our assumption that $X_i$ are i.i.d. is actually superfluous). Both proofs of Baxter and Wendel rely on the corresponding properties of {\it deterministic} sequences. Sparre Andersen's original proof of~\eqref{!!} does not allow such a nice description as it combines a simple combinatorial argument with some clever counting which rests on additivity of probability. The widely-known proof of this result given by Feller~\cite[Sec. XII.6]{Feller} offers much clearer combinatorial approach but heavily uses the independence of increments.


\subsection{Applications to intrinsic volumes of convex hulls}\label{1324}
Let us give some corollaries of Theorem~\ref{second}. In this subsection we always assume that $S_n$ is a random walk that satisfies \eqref{cond H}, and impose no other conditions. As in \eqref{1454}, let $S^{(1)}_n, \dots, S^{(d)}_n$ be independent copies of the  walk $S_n$.

First of all, by considering $g(x_1,\dots, x_{d-1})\equiv1$ in \eqref{1454}, we obtain the formula~\eqref{E faces} for the expected number of faces of the convex hull of $S_n$. Less trivial applications are as follows.

\begin{cor}[Expected surface area] \label{cor: E surface}
We have that
$$
\E\Vol_{d-1}(\partial C_n)=\frac{2}{(d-1)!}\sum_{\substack{j_1+\dots+j_{d-1}\leq n\\ j_1,\dots, j_{d-1}\geq1}}\frac{\E\, \det^{1/2}\Bigl(\bigl\langle S^{(m)}_{j_m},S^{(\ell)}_{j_\ell} \bigr\rangle\Bigr)_{m, \ell=1}^{d-1} } {j_1\cdot ... \cdot j_{d-1}}.
$$
\end{cor}

For $d=2$ this gives a formula by Spitzer and Widom~\cite{SW} on the average perimeter:
\be \label{perimeter}
\E \Vol_1 (\partial C_n) = 2 \sum_{j=1}^n \frac{\E \|S_j\|}{j}.
\ee
The three-dimensional version of this result was first obtained by Barndorff-Nielsen and Baxter~\cite{Nielsen}.

\begin{proof}
Applying \eqref{1454} with the Gram determinant formula
\begin{equation} \label{eq: Gram}
g(x_1,\dots, x_{d-1})=\Vol_{d-1}(\conv(0,x_1,\dots,x_{d-1}))=\frac{1}{(d-1)!} \sqrt{\det\left(\langle x_m, x_l\rangle\right)_{m,l=1}^{d-1}}
\end{equation}
with $x_1, \dots, x_{d-1} \in \R^d$, we get 
\begin{multline*}
\E\Vol_{d-1}(\partial C_n)=2\sum_{\substack{j_1+\dots+j_{d-1}\leq n\\ j_1,\dots, j_{d-1}\geq1}}\frac{\E\Vol_{d-1} \left( \conv\bigl(0,S^{(1)}_{j_1}, \dots, S^{(d-1)}_{j_{d-1}}\bigr)\right)}{j_1\cdot ... \cdot j_{d-1}} \\
=\frac{2}{(d-1)!}\sum_{\substack{j_1+\dots+j_{d-1}\leq n\\ j_1,\dots, j_{d-1}\geq1}}\frac {\E\, \det^{1/2}\Bigl(\bigl\langle S^{(m)}_{j_m},S^{(\ell)}_{j_\ell} \bigr\rangle\Bigr)_{m, \ell=1}^{d-1} } {j_1\cdot ... \cdot j_{d-1}}.
\end{multline*}
\end{proof}

\begin{cor}[Expected volume]\label{1745}
We have that
\be\label{1130}
\E\Vol_d (C_n)=\frac{1}{ d!}\sum_{\substack{j_1+\dots+j_{d}\leq n\\ j_1,\dots, j_{d}\geq1}}\frac{\E\left |\det\left[S^{(1)}_{j_1}, \dots, S^{(d)}_{j_{d}}\right] \right|}{j_1\cdot ... \cdot j_{d}}.
\ee
\end{cor}
A version of this result was first obtained by Barndorff-Nielsen and Baxter~\cite{Nielsen}.
\begin{proof}
Denote by $': \R^{d+1} \to \R^d$ the projection onto the first $d$ coordinates. Let $\tilde S_n$ be any $(d+1)$-dimensional random walk such that $\tilde S_n'=S_n$ and its last coordinate is distributed continuously and independently of $S_n$. The convex hull $\tilde C_n$ of $\tilde S_n$ satisfies $(\partial \tilde C_n)' = C_n$, and the pre-image under $'$ of any point from $\Intr (C_n)$ consists of exactly two points. Together with \eqref{boundary} this gives 
\begin{equation} \label{eq: projected}
2 \Vol_d (C_n) = \sum_{f \in \tilde{\mathcal{F}}_n} \Vol_d(f'),
\end{equation}
where $\tilde{\mathcal{F}}_n$ denotes the set of faces of $\tilde C_n$. 

Each face $f \in \tilde{\mathcal{F}}_n $ a.s. is a $d$-dimensional simplex  in $\R^{d+1}$ with vertices $\tilde{S}_{i_1(f)}, \dots, \tilde{S}_{i_{d+1}(f)}$ and so $f'$ a.s. is a $d$-dimensional simplex in $\R^d$. Its volume is given by
$$\Vol_d(f') = g(\tilde{S}_{i_2(f)} - \tilde{S}_{i_1(f)}, \dots, \tilde{S}_{i_{d+1}(f)} - \tilde{S}_{i_1(f)}),$$ 
where 
$$g(x_1,\dots,x_d)=\Vol_d(\conv(0,x'_1,\dots,x'_d))=\frac{1}{d!} \bigl |\det[{x'_1},\dots, {x'_d}] \bigr|$$
is defined for $x_1, \dots, x_d \in \R^{d+1}$. The claim then follows by combining \eqref{eq: projected} with \eqref{1454} applied for $\tilde{S}_n$ and the above given $g$.
\end{proof}

The following approach unifies the examples considered above. The volume and the surface area of a convex set are the special cases of so called \emph{intrinsic volumes} $V_0,\dots,V_d$, which naturally arise as the coefficients in the {\it Steiner formula}: for any convex set $K\subset\R^d$,
$$
\Vol_d(K+rB_d)=\sum_{k=0}^d \kappa_{d-k} V_k(K) r^{d-k}, \quad r\geq 0,
$$
where $B_d$ denotes a $d$-dimensional unit ball and $\kappa_k:=\pi^{k/2}/\Gamma(\frac k 2 +1)$ is the volume of $B_k$. In particular, $V_0(K)=1, V_d(K)=\Vol_d(K)$, $2 V_{d-1}(K)=\Vol_{d-1}(\partial K)$, and $V_1(K)$ equals the mean width of $K$ divided by the constant $\frac{2 \kappa_{d-1}}{d \kappa_d}$, which is the mean width of a unit segment in $\R^d$. The last statement readily follows from the other definition of intrinsic volumes, which sometimes is called the {\it Crofton formula}:
\begin{equation}\label{1717}
V_k(K):=\binom{d}{k}\frac{\kappa_d}{\kappa_{k}\kappa_{d-k}}\int_{\mathcal L_k^d}\Vol_k (K|L)\,\,d\mu_k(L),
\end{equation}
where $\mathcal L_k^d$ is the Grassmannian of all $k$-dimensional linear subspaces of $\R^d$ equipped with the Haar probability measure $\mu_k$, and $K|L$ is the orthogonal projection of $K$ onto $L$. 

Intuitively, the $k$-th intrinsic volume of $K$ equals, up to the constant factor, the mean $k$-dimensional volume of the projection of $K$ onto a uniformly chosen random $k$-dimensional linear subspace of $\R^d$. The normalization constant $\binom{d}{k}\frac{\kappa_d}{\kappa_{k}\kappa_{d-k}}$ is chosen so that the intrinsic volumes of $K$ do not depend on whether we consider $K$ as a subset of $\R^d$ or embed it in any higher-dimensional Euclidian space. For an extensive account on integral geometry we refer the reader to the books Santal\'o~\cite{lS76} and Schneider and Weil~\cite{schneider_weil_book}.

\begin{cor}[Expected intrinsic volumes]
We have
$$
\E V_k (C_n) = \frac{1}{k!}\sum_{\substack{j_1+\dots+j_k\leq n\\ j_1,\dots, j_k\geq1}}\frac{\E\, \det^{1/2}\Bigl(\bigl\langle S^{(m)}_{j_m},S^{(\ell)}_{j_\ell} \bigr\rangle\Bigr)_{m, \ell=1}^{d-1} } {j_1\cdot ... \cdot j_k},\quad k=1,\dots,d.
$$
In particular, the Spitzer--Widom formula~\eqref{perimeter} naturally extends to any dimension:
$$
\E V_1 (C_n) = \sum_{j=1}^n \frac{\E\|S_j\|}{j}.
$$
\end{cor}
These results were already used by Molchanov and Wespi~\cite[Theorem 2.3]{MolchanovWespi} to compute intrinsic volumes of the closed convex hull of a symmetric $\alpha$-stable L{\'e}vy process in $\R^d$ with $\alpha \in (1,2]$. For a standard Brownian motion, the intrinsic volumes $V_1$ and $V_2$ were found in the earlier paper by Kampf et al.~\cite{kampf_etal}.

\begin{proof}
For any $L \in \mathcal L_k^d$, the sequence $\tilde S_n:= S_n|L, n \ge 0,$ is a $k$-dimensional random walk satisfying \eqref{cond H} and its convex hull is $\tilde C_ n = C_n|L$. Hence by Corollary~\ref{1745}, one has
\begin{eqnarray*}
\E\Vol_k (C_n|L)&=&\frac{1}{ k!}\sum_{\substack{j_1+\dots+j_{k}\leq n\\ j_1,\dots, j_{k}\geq1}}\frac{\E\left |\det\left[\tilde{S}^{(1)}_{j_1}, \dots, \tilde{S}^{(d)}_{j_{k}}\right] \right|}{j_1\cdot ... \cdot j_{d}} \\
&=&\sum_{\substack{j_1+\dots+j_k\leq n\\ j_1,\dots, j_k\geq1}}\frac{\E\Vol_k \Bigl( \conv\bigl(0,S^{(1)}_{j_1}, \dots, S^{(k)}_{j_k}\bigr) \bigl | \bigr. L \Bigr)}{j_1\cdot ... \cdot j_k}.
\end{eqnarray*}
Integrate this equation over $\mathcal L_k^d$ with respect to $\mu_k$, normalize according to the definition of the intrinsic volume, and apply the Fubini theorem to both sides to get 
$$
\E V_k (C_n)= \sum_{\substack{j_1+\dots+j_k\leq n\\ j_1,\dots, j_k\geq1}}\frac{\E V_k \bigl( \conv\bigl(0,S^{(1)}_{j_1}, \dots, S^{(k)}_{j_k}\bigr)  \bigr)}{j_1\cdot ... \cdot j_k}.
$$
The linear dimension of $K= \conv\bigl(0,S^{(1)}_{j_1}, \dots, S^{(k)}_{j_k}\bigr)$, which is a convex hull of $k+1$ points, is $k$ a.s., hence $V_k(K)=\Vol_k(K)$ and the claim follows if we use the Gram determinant formula~\eqref{eq: Gram}.
\end{proof}

\section{Applications of Gaussian convex hulls to geometry}\label{1330}
In this section we always assume that $X_1, \dots, X_n$ are independent standard Gaussian vectors in $\R^d$. 

\subsection{Intrinsic volumes of canonical orthoschemes.} Consider the Gaussian random $d \times n$ matrix $A$ with the columns $X_1, \dots, X_n$. Its rows $Y_1, \dots , Y_d$ are standard Gaussian vectors in $\R^n$.
It is known that the linear span of $Y_1, \dots , Y_d$ (which are in general position with probability one) is a random $d$-dimensional linear subspace of $\R^n$ uniformly distributed on the Grassmannian $\mathcal{L}_d^n$ with respect to the Haar
probability measure. Using this fact and the Crofton formula~\eqref{1717}, it can be shown that for any convex body $K\subset\R^n$
$$
V_d(K)=\frac{(2\pi)^{d/2}}{d!\kappa_d}\E\Vol_d(\conv(\{Ax\,:\,x\in K\})).
$$
This equation is a finite-dimensional version of a general result of Sudakov~\cite{vS76} (for $d=1$) and Tsirelson~\cite{tsirelson_1, tsirelson_2, tsirelson_3} (for general $d$) on Gaussian measures in infinite-dimensional spaces.

Consider the simplex $T_n \subset \R^n$ with vertices
$$
(0,0,\dots,0),(1,0,\dots,0),(1,1,0,\dots,0),\dots,(1,\dots,1),
$$
which we call \emph{Schl\"afli canonical orthoscheme}. Such simplexes are also called {\it path-simplexes}.

Now
$$
\conv(\{Ax\,:\,x\in T_n\})=C_n,
$$
which implies that
\be\label{1634}
V_d(T_n)=\frac{(2\pi)^{d/2}}{d!\kappa_d}\E\Vol_d(C_n).
\ee
Combining this equality with~\eqref{1130}, we obtain 
$$
V_d(T_n)=\frac{(2\pi)^{d/2}}{(d!)^2\kappa_d}\sum_{\substack{j_1+\dots+j_{d}\leq n\\ j_1,\dots, j_{d}\geq1}}\frac{\E\left |\det\left[S^{(1)}_{j_1}, \dots, S^{(d)}_{j_d}\right] \right|}{j_1\cdot ... \cdot j_{d}},
$$
where $S^{(1)}_1, \dots, S^{(d)}_d$ are independent standard Gaussian random walks in $\R^d$. Let $M$ be a $d \times d$ matrix with independent standard normal entries. Then $\E |\det M| = \E \sqrt{\det (M M^\top)}$, where $M M^\top$ is a Wishart matrix whose determinant has a well known distribution and moments. Hence (see for example Kabluchko and Zaporozhets~\cite{KZ12})
$$
\E\left |\det\left[S^{(1)}_{j_1}, \dots, S^{(d)}_{j_{d}}\right] \right|= \frac{d!\kappa_d}{(2\pi)^{d/2}}\sqrt{j_1\cdot ... \cdot j_{d}},
$$
which implies that
\be\label{1449}
V_d(T_n)=\frac{1}{d!}\sum_{\substack{j_1+\dots+j_{d}\leq n\\ j_1,\dots, j_{d}\geq1}}\frac{1}{\sqrt{j_1\cdot ... \cdot j_{d}}}.
\ee

This result was first obtained by Gao and Vitale~\cite{GV01}, who considered a direct geometric approach using a formula for intrinsic volumes of convex polytopes. The simplex $T_n/\sqrt{n} \in \R^n$ is a finite-dimensional approximation of the closed convex hull $T$ of a {\it Wiener spiral\,}\footnote{This is the deterministic curve $\{ W(t), t \in [0,1] \}$, where $W$ is a standard Wiener process, in the Hilbert space of square-integrable zero mean random variables.} in a Hilbert space, which was introduced by Kolmogorov in 1940; \cite{GV01} calls $T$ the Brownian motion body. Note that $T$ is isometric to the subset of non-increasing functions of $L_2[0,1]$ that take values in $[0,1]$. Gao and Vitale~\cite{GV01} used \eqref{1449} to prove that
$$
V_d(T)=\frac{\kappa_d}{d!}.
$$
Due to Tsirelson~\cite{tsirelson_1, tsirelson_2, tsirelson_3}, the normalized $d$-th intrinsic volume of $T$ is equal to expected volume of the convex hull of a $d$-dimensional Brownian motion, see Kabluchko and Zaporozhets~\cite{KZ14} for details. The latter quantity was calculated by Eldan~\cite{rE14} using direct methods.

\subsection{Spherical intrinsic volumes of canonical orthoschemes.}
Let us consider the unit sphere $\S^n$ in $\R^{n+1}$. By saying that $K \subset \S^n$ is convex we mean that the conic hull of $K$ in $\R^{n+1}$ is convex and line-free. Following Santal\'o (see~\cite[Section~IV.4]{lS76}),  for a convex body $K$ in $\S^n$ we can use a spherical counterpart of the Crofton formula~\eqref{1717} to define
$$
U_k(K):=\frac12\int_{S^n_{n-k}}\mathbbm{1}_{\{K\cap s \neq \varnothing\}} \,d\nu_{n-k}(s),
$$
where $S^n_k$ denotes the space of $k$-dimensional great subspheres of $\S^n$ equipped with the rotationally invariant probability measure $\nu_k$. The functionals $U_k$ can be considered as {\it spherical} counterparts of   Euclidean intrinsic volumes $V_k$. However, there are other possible definition of spherical intrinsic volumes. For basic facts from spherical integral geometry we refer the reader to Gao et al.~\cite{GHS03}, McCoy and Tropp~\cite{MT14}, and Schneider and Weil~\cite[Sec. 6.5]{schneider_weil_book}.

Similarly to \eqref{1634}, it can be shown (see G\"otze at al.~\cite{GKZ15} for details) that
$$
U_d(\tilde T_n)=\frac12\P(0\in C_n),
$$
where $\tilde T_n$ denotes the intersection of the conic hull of $T_n$ with $\S^{n-1}$. To eliminate any misunderstanding, by the spherical intrinsic volumes of the canonical orthoscheme $T_n$ mentioned in the abstract and Section~\ref{Sec: Intro} we meant exactly $U_d(\tilde T_n)$. It follows from~\eqref{!!} and Theorem~\ref{main} that
$$
U_1(\tilde T_n)=\frac12-\frac{(2n-1)!!}{(2n)!!},\quad U_2(\tilde T_n)=\frac12-\sum_{k=1}^n \frac{(2n-2k-1)!!}{2k \cdot (2n-2k)!!}.
$$
As we explained in the introduction, the exact values  of the other spherical intrinsic volumes of $\tilde T_n$ are not accessible by the method of this paper. They are available in our most recent  work~\cite{KVZ15} coauthored with Z. Kabluchko, where $U_k$ are referred to as half-tail functionals. 

\section{Combinatorial arguments}\label{1333}

For any $x_1, \dots, x_n \in \R^d$, denote by
$$
s_0:=0, \, s_k:=x_1+\dots + x_k, \quad  k=1,\dots,n,
$$
the sequence of partial sums. For any permutation $\sigma=(\sigma(1),\dots,\sigma(n))$, denote
$$
s_0(\sigma) := 0, s_k(\sigma):=x_{\sigma(1)}+\dots+x_{\sigma(k)}, \,\, k=1,\dots, n.
$$


We first proved a simple combinatorial statement, which generalizes two-dimensional Lemma~1 from Baxter~\cite{Baxter} to higher dimensions. Later we found this result in the paper by Barndorff-Nielsen and Baxter~\cite{Nielsen}. The one-dimensional version is known as the ``cycle lemma'', for example, see Steele~\cite[Section 4]{Steele2002} and references therein for further combinatorial applications. For the reader's convenience we present the proof here.

\begin{lem}\label{Lemma comb}
Let $x_0, x_1,\dots,x_n\in \R^d$, and let $H$ be a closed half-space such that
$$
x_0, x_0+ s_n \in \partial H \quad \mbox{and} \quad x_0 + s_j - s_i \notin \partial H, \, 0\le i < j \le n-1.
$$
There exists exactly one cyclic permutation $\sigma=(k+1,\dots,n,1,\dots,k)$ such that
$$
x_0, x_0 + s_1(\sigma),\dots, x_0 + s_n(\sigma)\in H.
$$
\end{lem}

\begin{proof}
By the asumption, there exists exactly one point $x_0 + s_k$ among $\{ x_0 + s_i \}_{i=0}^{n-1} \cap (\Intr(H))^c$ that is
at the maximum distance (possibly zero) from $\partial H$. Then $\sigma:=(k+1,\dots,n,1,\dots,k)$ is a required permutation, and it is unique by the uniqueness of $k$.
\end{proof}

Our next goal is to obtain stochastic versions of this result. For any points $x_1,\dots, x_d \in \R^d$, define
$$H_\pm(x_1,\dots, x_d):=\{z \in \R^d: \pm \det [ x_2 - x_1, \dots,  x_d - x_1, z - x_1 ] \ge 0 \}.$$
If there is a unique hyperplane through these points, then this definition gives a rule to distinguish between the two half-spaces $H_+$ and $H_-$ lying on different sides of the hyperplane. If such a hyperplane is not unique, then $H_\pm = \R^d.$

\begin{lem}\label{Lemma prob}
Assume that the partial sums $S_k$ of $n$-exchangeable random vectors $X_1, \dots, X_n$ in $ \R^d, d \ge 2,$ satisfy \eqref{cond Gen}. For any indices $1 \le i_1 < \dots < i_{d-2} \le n-1$, we have
$$\P \bigl(S_1, \dots, S_n \in H_\pm(0, S_{i_1}, \dots, S_{i_{d-2}}, S_n) \bigr) = \frac{1}{i_1 (i_2-i_1) \cdot ... \cdot (n -i_{d-2})},$$ and moreover,
$$\P \bigl(S_1, \dots, S_n \in H_\pm(0, S_{i_1}, \dots, S_{i_{d-2}}, S_n) \bigl| \bigr. S_{i_1}, \dots, S_{i_{d-2}}, S_n \bigr) = \frac{1}{i_1 (i_2-i_1) \cdot ... \cdot (n -i_{d-2})} \text{ a.s.}$$
\end{lem}

This is a little generalization of the well-known fact that the trajectory of any continuously distributed one-dimensional random walk $S_n$ lies above the line joining $(0,0)$ and $(n, S_n)$ with probability $1/n$, see Feller~\cite[Sec. XII.9]{Feller}. The fact follows from Lemma~\ref{Lemma prob} if we consider the two-dimensional walk $\tilde{S}_n:=(n, S_n)$ with deterministic first component.

\begin{proof}
With probability one, there exists a unique half-plane through $0, S_{i_1}, \dots, S_{i_{d-2}}, S_n$ (otherwise we could add any other point $S_k$ and arrive at a contradiction with \eqref{cond Gen}.) Hence almost surely,
$$H_\pm(S):=H_\pm(0, S_{i_1}, \dots, S_{i_{d-2}}, S_n)$$ are half-spaces. 

For any permutation $\sigma=(\sigma(1),\dots,\sigma(n))$, introduce the partial sums
$$
S_0(\sigma) := 0, \, S_k(\sigma):=X_{\sigma(1)}+\dots+X_{\sigma(k)}, \quad k=1,\dots, n.
$$
Put $i_0:=0, i_{d-1}:=n$ and denote by $\mathcal S$ the set of $(i_1 - i_0) \cdot ... \cdot (i_{d-1} -i_{d-2})$ permutations of length $n$ that are products over $j$ from $1$ to $d-1$ of cyclic permutations of the form
\be \label{cyclic}
(k_j+1, \dots, i_j, i_{j-1} + 1, \dots, k_j),
\ee
where $i_{j-1} + 1 \le k_j \le i_j$. Note that any $\sigma \in \mathcal S$ does not change $H_\pm$, i.e. $H_\pm(S)=H_\pm(S(\sigma))$, since for every $k \in \{ i_1, \dots, i_{d-2}, n\}$ one has $S_k = S_k(\sigma)$, and the sequences of partial sums $S$ and $S(\sigma)$ have the same distribution by the exchangeability of the increments.

For any $0 \le j \le d-2$, the random vectors $S_{i_j}, X_{i_j + 1}, \dots, X_{i_{j+1}}$ and the half-space $H_\pm(0, S_{i_1}, \dots, S_{i_{d-2}}, S_n)$ satisfy the assumption of Lemma~\ref{Lemma comb} with probability one. Indeed, if for some $i_j \le m < \ell < i_{j+1}$, one has $S_{i_j} + S_\ell - S_m \in \partial H_\pm(S)$ with positive probability, then among the partial sums $S_k(\sigma)$ with 
$$\sigma=(1, \dots, i_j, m + 1, \dots, \ell, i_j+1, i_j +m, \ell +1, \dots, n)$$ there are $d$ points $S_{i_1}(\sigma), \dots, S_{i_{d-1}}(\sigma), S_{i_j+\ell-m}(\sigma)$ that belong to the hyperplane $\partial H_\pm$ passing through $0$, which contradicts \eqref{cond Gen} by the exchangeability of  increments.

By Lemma~\ref{Lemma comb}, there exists an a.s. unique random permutation $\sigma_\pm=\sigma_\pm(S) \in \mathcal S$ such that $S_1(\sigma_\pm), \dots, S_n(\sigma_\pm) \in H_\pm(S(\sigma_\pm)) = H_\pm(S)$. Hence the sum in r.h.s. of the equality
\be \label{perm} \P \bigl(S_1, \dots, S_n \in H_\pm(S_{i_0}, S_{i_1}, \dots, S_{i_{d-1}}) \bigr) = \frac{1}{|\mathcal{S}|} \E \biggl [ \sum_{\sigma \in \mathcal S} \I \bigl(S_1(\sigma), \dots, S_n(\sigma) \in H_\pm(S) \bigr) \biggr ]
\ee
equals one a.s.. This proves the first assertion of the lemma. Similarly, for any non-negative Borel function $g:\R^{d \times (d-1)} \to \R$, we have
\beaa
&&\E \Bigl[ g(S_{i_1}, \dots, S_{i_{d-2}}, S_n) \I \bigl(S_1, \dots, S_n \in H_\pm(S_{i_0}, S_{i_1}, \dots, S_{i_{d-1}}) \bigr) \Bigr]\\
&=& \frac{1}{|\mathcal{S}|} \E \Bigl [ {g(S_{i_1}, \dots, S_{i_{d-2}}, S_n)} \sum_{\sigma \in \mathcal S} \I \bigl(S_1(\sigma), \dots, S_n(\sigma) \in H_\pm(S) \bigr) \Bigr]\\
&=& \frac{1}{|\mathcal{S}|} \E g(S_{i_1}, \dots, S_{i_{d-2}}, S_n),
\eeaa
and the second claim of the lemma follows by the definition of conditional expectation.
\end{proof}

We conclude this section with a result on random bridges. 

\begin{lem} \label{Lemma bridge}
Let $S_k$ be a random bridge of length $n+1$ in $\R^d, d \ge 2,$ such that $S_1, \dots, S_n$ satisfy~\eqref{cond Gen}. For any indices $0 \le i_1 < \dots < i_{d-1} \le n $, we have
$$\P \bigl(S_1, \dots, S_n \in H_\pm(0, S_{i_1}, \dots, S_{i_{d-1}}) \bigr) = \frac{1}{i_1 (i_2-i_1) \cdot ... \cdot (n -i_{d-1}+1)}.$$
The above also holds true if $S_k$ is either the difference bridge or a well-defined conditional bridge of a random walk in $\R^d$ that satisfies~\eqref{cond H}. 
\end{lem}

This is a multidimensional counterpart of the fact that in dimension one, a bridge of length $n$   of a continuously distributed random walk stays positive with probability $1/n$.

\begin{proof}
As we already mentioned in Section~\ref{1306}, by our understanding of the conditioning it is clear that a conditional random walk bridge satisfies~\eqref{cond Gen} if the increments of the underlying random walk satisfies~\eqref{cond H}. It is also easy to see that the difference bridge of such random walk satisfies~\eqref{cond Gen}. Thus the latter assumptions holds true in all cases.

By repeating the argument used in the proof of Lemma~\ref{Lemma prob}, we see that \eqref{perm} holds for $i_0=0$ and $\mathcal S$ defined to be the set of permutations of length $n+1$ that are products of $d$ cyclic permutations of the form \eqref{cyclic}, where $0 \le j \le d-1$ and $i_d=n+1$.
\end{proof}

\section{Proofs} \label{prop_proof}

{\bf Proof of Proposition~\ref{face_prob}.} Recall that $0 \leq i_1<\dots<i_d\leq n$. By \eqref{boundary} and \eqref{a face},
\begin{multline} \label{pm} 
\P(\conv(S_{i_1},  \dots, S_{i_d}) \in \mathcal F_n) =  \P(0, S_1, \dots, S_n \in H_+(S_{i_1}, \dots, S_{i_d}) )  \\
+ \P(0, S_1, \dots, S_n \in H_-(S_{i_1}, \dots, S_{i_d}) ). 
\end{multline}
Denote $$H_\pm:=H_\pm(0, S_{i_2} - S_{i_1}, \dots, S_{i_d} - S_{i_1}),$$
then by $S_{i_d} - S_{i_1} \in \partial H_\pm$, we have
\begin{equation} \label{eq: hyperplanes}
H_+(S_{i_1}, \dots, S_{i_d}) = H_\pm + S_{i_1} = H_\pm + S_{i_d}. 
\end{equation}

\underline{Proof of \eqref{pinned bridge}.} Here $S_k$, $1 \le k \le n+1$, with $S_{n+1}:=0$, is a random bridge of length $n+1$. Let us consider the transformation that translates the whole trajectory of the bridge by moving the origin to $S_{i_1}$.  
The transformed trajectory corresponds to the random bridge of partial sums $S_k(\sigma)$, $1 \le k \le n+1$, with $\sigma=(i_1+1, \dots, n +1, 1, \dots, i_1)$. By the first equality in \eqref{eq: hyperplanes}, we have
\begin{eqnarray*}
&&\bigl \{ 0, S_1, \dots, S_n \in H_\pm (S_{i_1}, \dots, S_{i_d}) \bigr \} \\
&=& \bigl \{ S_{n+1} - S_{i_1}, S_{n+1} - S_{i_1} + S_1, \dots, 0, S_{i_1 +1 } - S_{i_1}, \ldots, S_n - S_{i_1} \in H_\pm  \bigr \}\\
&=& \bigl \{ S_{n+1 - i_1}(\sigma), S_{n+2 - i_1}(\sigma), \dots, 0, S_1(\sigma), \ldots, S_{n-i_1}(\sigma) \in H_\pm(0, S_{i_2 - i_1}(\sigma), \ldots,  S_{i_d - i_1}(\sigma))  \bigr \},
\end{eqnarray*}
hence
$$\P(0, S_1, \dots, S_n \in H_\pm (S_{i_1}, \dots, S_{i_d}) ) = \P(S_1, \dots, S_n \in H_\pm (0, S_{i_2 - i_1}, \dots, S_{i_d - i_1}) ),$$
and then \eqref{pinned bridge} follows by \eqref{pm} and Lemma~\ref{Lemma bridge}.

\underline{Proof of  \eqref{pinned}.} Let us split the trajectory of $S_n$ into three parts: by \eqref{eq: hyperplanes}, we have
\begin{align}
&\bigl \{0, S_1, \dots, S_n \in H_\pm(S_{i_1}, \dots, S_{i_d}) \bigr \} \notag \\
&=\bigl \{0, S_1, \dots, S_{i_1} \in H_\pm + S_{i_1} ; \, S_{i_1 + 1}, \dots, S_{i_d} \in H_\pm + S_{i_1} \, ; \, S_{i_d +1}, \dots, S_n  \in H_\pm + S_{i_d} \bigr \}. \label{three parts}
\end{align}
Since $S_n$ is a random walk, by conditioning on $X_{i_1+1}, \dots, X_{i_d}$, which define $H_\pm$, and using the independence of increments, we obtain that
\begin{multline}
\label{eq: three parts}
\P \bigl (0, S_1, \dots, S_n \in H_\pm(S_{i_1}, \dots, S_{i_d}) \bigr ) \\
= \P \bigl (-S_{i_1}', -S_{i_1 - 1}', \dots, 0 \in H_\pm \bigr) 
\P \bigl( S_{i_1 + 1} - S_{i_1}, \dots, S_{i_d} - S_{i_1}\in H_\pm \bigr) \P \bigl( S_1', \dots, S_{n - i_d}'  \in H_\pm \bigr)  \\
= \P \bigl (S_1', \dots, S_{i_1}' \in H_\mp \bigr) \P \bigl( S_1', \dots, S_{n - i_d}'  \in H_\pm \bigr) \\
\times \P \bigl( S_1, \dots, S_{i_d - i_1} \in H_\pm(0, S_{i_2 - i_1}, \dots, S_{i_d - i_1}) \bigr), 
\end{multline}
where $S_n'$ is an independent copy of the random walk $S_n$. Let us stress that we obtained~\eqref{eq:  three parts} assuming that $S_n$ is a random walk satisfying~\eqref{cond H} but not~\eqref{cond Sym}. Then \eqref{pinned} holds by \eqref{pm}, Lemma~\ref{Lemma prob}, and the following simple result.

\begin{lem}\label{Lemma half}
Let $S_n$ be a random walk in $\R^d$, and let $H$ be a half-space such that $0\in\partial H$. Assume that \eqref{cond H} and \eqref{cond Sym} hold. Then
$$
\P(S_1,\dots, S_n \in H)=\frac{(2n-1)!!}{(2n)!!}.
$$
\end{lem}
\begin{proof}
Denote by $u=u_H$ the unit vector that is orthogonal to $\partial H$ and belongs to $H$. The distribution of increments of the one-dimensional random walk $S_k^{(u)}:=\langle S_k, u\rangle, k \ge 1,$ is continuous and symmetric, hence the result follows by \eqref{!!} and
$$
\P(S_1,\dots, S_n\in H)=\P(S_1^{(u)} > 0,\dots, S_n^{(u)}>0).
$$
\end{proof}

\underline{Proof of \eqref{temporal}.} If $i_1 \neq 0$, we transform the trajectory by interchanging its part from $1$ to $i_1$ with the part from $i_1+1$ to $i_d$; this does not change the part from $i_d+1$ to $n$. See Figure~\ref{fig}, where the parts are denoted by ${\bf T_1}$, ${\bf T_2}$, and ${\bf T_3}$, respectively. The key observation is that for the transformed trajectory, $S_{i_d}$ becomes the most distant point from $H_\pm = H_\pm(0, S_{i_2} - S_{i_1}, \dots, S_{i_d} - S_{i_1})$. Let us prove this formally.

\begin{figure}
\begin{center}
\begin{minipage}{.5\textwidth}
\setlength{\parindent}{-5ex}
  \includegraphics{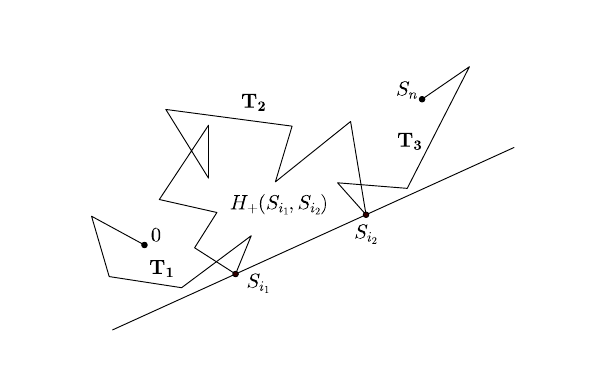}
\par  
\end{minipage}%
\begin{minipage}{.5\textwidth}
\setlength{\parindent}{-5ex}
  \includegraphics{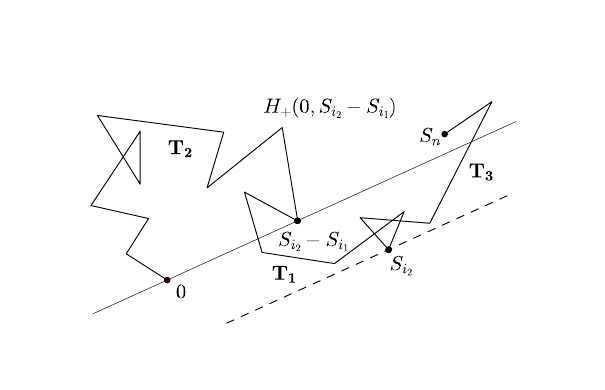}
\par
\end{minipage}
\vskip -1em
\caption{The path transform for $d=2$.}
\label{fig}
\end{center}
\end{figure}

If $i_1 \neq 0$, we rewrite the event $\{0, S_1, \dots, S_n \in H_\pm(S_{i_1}, \dots, S_{i_d}) \}$ in terms of the partial sums $S_k(\sigma), 1 \le k \le n,$ with
$$\sigma=(i_1+1, \dots, i_d, 1, \dots, i_1, i_d+1, \dots, n).$$
For the parts ${\bf T_1}$ and ${\bf T_3}$ of the trajectory, we use \eqref{eq: hyperplanes} to obtain
\begin{multline*}
\bigl \{ 0, S_1, \dots, S_{i_1}, S_{i_d +1}, \dots, S_n \in H_\pm(S_{i_1}, \dots, S_{i_d}) \bigr \}\\
=\bigl \{ 0, S_1, \dots, S_{i_1} \in H_\pm + S_{i_1} \, ; \, 0, S_{i_d +1}, \dots, S_n \in H_\pm + S_{i_d} \bigr \} \\
= \bigl \{ S_{i_d} - S_{i_1}, S_{i_d} - S_{i_1} + S_1, \dots, S_{i_d}, 0, S_{i_d +1}, \dots, S_n  \in H_\pm + S_{i_d} \bigr \} \\
= \bigl \{ 0, S_{i_d - i_1 +1}(\sigma), \dots, S_n(\sigma) \in H_\pm(0, S_{i_2 - i_1}(\sigma), \dots, S_{i_d - i_1}(\sigma)) + S_{i_d}(\sigma) \bigr \}.
\end{multline*}
Note that the event in the second line differs from the corresponding part of \eqref{three parts} since we added $0$ to second group of variables. For the part ${\bf T_2}$, 
\begin{multline*}
\bigl \{ S_{i_1 + 1}, \dots, S_{i_d} \in H_\pm(S_{i_1}, \dots, S_{i_d})  \bigr \} \\
=  \bigl \{ S_1 (\sigma), \dots, S_{i_d - i_1}(\sigma) \in H_\pm(0, S_{i_2 - i_1}(\sigma), \dots, S_{i_d - i_1}(\sigma)) \bigr \}.
\end{multline*}

Combining the above gives
\begin{align}
\P \bigl (0, S_1, \dots, S_n \in H_\pm(S_{i_1}, \dots, S_{i_d}) \bigr ) 
= \P &\Bigl ( S_1, \dots, S_{i_d - i_1} \in H_\pm(0, S_{i_2 - i_1}, \dots, S_{i_d - i_1}) ;  \label{eq: most distant}\\ 
& 0, S_{i_d - i_1 +1}, \dots, S_n \in H_\pm(0, S_{i_2 - i_1}, \dots, S_{i_d - i_1}) + S_{i_d} \Bigr ). \notag
\end{align}
By projecting on the orthogonal compliment, we see that $S_{i_d}$ is a most distant point from $H_\pm(\dots)$ among $0, S_{i_d - i_1 +1}, \dots, S_n$. Such a point is a.s. unique by assumption \eqref{cond Gen}. Note that \eqref{eq: most distant} is also valid for $i_1=0$ since in this case $H_\pm + S_{i_d} = H_\pm$. Therefore, \eqref{eq: most distant} can be written for all $0 \le i_1 <i_2$ as 
\begin{align}
\P(0, S_1, \dots, S_n \in H_\pm(S_{i_1}, \dots, S_{i_d}) ) &= \P \Bigr(S_1, \dots, S_{i_d - i_1} \in H_\pm(0, S_{i_2-i_1}, \dots, S_{i_d-i_1}); \label{eq: argmin} \\
& \argmin \limits_{0 \le k \le n -(i_d-i_1)} \det[ S_{i_2-i_1}, \dots, S_{i_d-i_1} , S_{i_d-i_1+ k} ] = \{i_1\} \Bigr), \notag
\end{align}
For a fixed temporal structure, i.e., the tuple $(i_2-i_1, \dots, i_d-i_1)$, it remains to sum in \eqref{eq: argmin} over $i_1$ from $0$ to $n-(i_d-i_1)$. The argmin disappears, and \eqref{temporal}  follows by \eqref{pm} and Lemma~\ref{Lemma prob}.
Proposition~\ref{face_prob} is now proved.

{\bf Remark on Footnote~\ref{note}.} Note that under assumptions made, \eqref{!!} does hold true in the one-dimensional case, see Sparre Andersen~\cite[Theorem 1]{Sparre0} or~\cite[Theorem 4]{Sparre2}; strictly speaking, both theorems are stated under slightly stronger assumptions which actually can be weakened to fit our requirements. The latter theorem yields \eqref{pinned} when applied for the partial sums of $(n-i_d+i_1)$-exchangeable one-dimensional increments $X_k':=\pm \det[S_{i_2-i_1}, \dots, S_{i_d-i_1}, X_{i_d-i_1+k}]$ (where $1 \le k \le n-i_d+i_1$) with $n-i_d+i_1$ substituted for $n$ and the first event in the r.h.s. of \eqref{eq: argmin}  substituted for the $C_n$ of~\cite[Theorem 4]{Sparre2}.

{\bf Proof of Proposition~\ref{face_prob asympt}.} The main ingredient is the following asymptotic version of Lemma~\ref{Lemma half}. Recall that for any half-space $H$ of $\R^d$, by $u_H$ we denote the unit vector that is orthogonal to $\partial H$ and belongs to $H$, and conversely, for any non-zero $u \in \R^d$, we put $H(u) = \{z \in \R^d: \langle z, u \rangle \ge 0 \}$.

\begin{lem}
\label{lem: half asympt}
Let $S_k$ be a random walk in $\R^d, d \ge 1,$ with increments that have zero mean, a finite covariance matrix $\Sigma$, and satisfy  \eqref{cond H}. Then
\begin{equation}
\label{eq: exit halfspace}
\lim_{n \to \infty} \sqrt{n} \P(S_1,\dots, S_n \in H) = \sqrt{\frac{2}{\pi} } R(u_H)
\end{equation}
uniformly over all half-spaces $H$ of $\R^d$ such that $0 \in \partial H$. The limit function $R(u)$ is continuous and positive on $\S^{d-1}$.
\end{lem}

We will see that the pointwise convergence in~\eqref{eq: exit halfspace} holds by a simple reduction to the well-known one-dimensional result of fluctuation theory. The difficulty is in showing that the convergence is uniform. Since this is quite a technical statement and the main message of our paper is in combinatorial methods, we postpone the proof of Lemma~\ref{lem: half asympt} until the Appendix.

Let us conclude the proof of Proposition~\ref{face_prob asympt}. We first recall that the cross product of $x_1, \dots, x_{d-1} \in \R^d$ is given by $x_1 \times \ldots \times x_{d-1} = \sum_{k=1}^d \det[x_1, \dots, x_{d-1}, e_k] e_k,$ where $e_1, \dots, e_d$ is the standard basis of $\R^d$.  Now consider~\eqref{eq:  three parts} with $i_1=0$. The first probability in the last equation in~\eqref{eq:  three parts} does not appear. Conditioning on $S_{i_2}, S_{i_3}, \dots, S_{i_d}$, which determine 
$$u_{\pm} := \pm  S_{i_2 } \times \ldots \times S_{i_d} = \pm (S_{i_2} - S_{i_1}) \times \ldots \times (S_{i_d} - S_{i_{d-1}})$$ 
and thus fix $H(u_\pm) = H_\pm(0, S_{i_2}, \dots, S_{i_d})$, and using Lemma~\ref{Lemma prob} for the third probability in~\eqref{eq:  three parts}, we get
$$
\frac{\P \bigl (0, S_1, \dots, S_n \in H_\pm(0, S_{i_2}, \dots, S_{i_d}) \bigr )}{ i_2 (i_3-i_2) \cdot \ldots \cdot (i_d - i_{d-1}) } = \E \bigl [\P\bigl( S_1', \dots, S_{n - i_d}'  \in H(u_\pm)  \bigl| \bigr. u_\pm \bigr) \bigr]. 
$$

The r.h.s. is $O(\frac{1}{\sqrt{n-i_d+1}})$ by Lemma~\ref{lem: half asympt}, implying the required uniform boundedness of the $o(1)$ term in Proposition~\ref{face_prob asympt}. Applying Lemma~\ref{lem: half asympt} one more time gives that
\begin{equation}
\label{eq: uniform discrete 1}
\frac{\P \bigl (0, S_1, \dots, S_n \in H_\pm(0, S_{i_2}, \dots, S_{i_d}) \bigr ) }{ i_2 (i_3-i_2) \cdot \ldots \cdot (i_d - i_{d-1})} =  \Bigl(\sqrt{\frac{2}{\pi}}  + o(1) \Bigr)\frac{\E R \bigl(\pm S_{i_2 } \times \ldots \times S_{i_d - i_{d-1}} \bigr)}{\sqrt{n-i_d}} 
\end{equation}
uniformly in $1 \le i_2 < i_3 < \ldots < i_d \le n- h_n$, since by its definition, $R(u)$ is an angular function and $u_\pm$ is a.s. non-zero by assumption~\eqref{cond H}. Then 
$$\E R \bigl(\pm S_{i_2 } \times \ldots \times S_{i_d - i_{d-1}} \bigr) = \E R \biggl ( \pm  \frac{S_{i_2}^{(1)}}{\sqrt{i_2}} \times \frac{S_{i_3-i_2}^{(2)}}{\sqrt{i_3 - i_2}} \times \ldots \times \frac{S_{i_d - i_{d-1}}^{(d-1)}}{\sqrt{i_d-i_{d-1}}}  \biggr),$$ where, recall, that $S_k^{(1)}, \ldots, S_k^{(d-1)}$ are independent copies of the random walk $S_k$.

Let $N_1, \ldots, N_{d-1}$ be independent standard Gaussian random vectors in $\R^d$. Since $R(u)$ is a continuous bounded function on $\R^d \setminus \{0\}$ and the cross product $x_1 \times \ldots \times x_{d-1}$ is continuous on $\R^{d \times (d-1)}$, the central limit theorem combined with the continuous mapping theorem and the fact $N_1 \stackrel{d}{=} -N_1$ imply that
\begin{equation}
\label{eq: uniform discrete 2}
\lim_{n \to \infty} \E R \bigl(\pm S_{i_2 } \times \ldots \times S_{i_d - i_{d-1}} \bigr) = \E R \bigl( \Sigma^{1/2} N_1 \times \ldots \times \Sigma^{1/2} N_{d-1} \bigr) 
\end{equation} 
uniformly in $1 \le i_2 < i_3 < \ldots < i_d \le n$ such that $\min(i_2, i_3 - i_2, \ldots, i_d - i_{d-1}) \ge h_n$. 

Finally, for any non-degenerate $d \times d$ matrix $A$ one has
\begin{multline*}
A N_1 \times \ldots \times A N_{d-1} = \sum_{k=1}^d \det[A N_1, \dots, A N_{d-1}, e_k] e_k \\
= \sum_{k=1}^d \det(A \cdot [N_1, \dots, N_{d-1}, A^{-1} e_k]) e_k 
= \det A \cdot (A^{-1})^T (N_1 \times \ldots \times N_{d-1}).
\end{multline*}
In particular, this shows that the distribution of $N_1 \times \ldots \times N_{d-1}$ is invariant under orthogonal transformations since the standard Gaussian distribution is so. Hence the angular component of this distribution is uniform on $\S^{d-1}$ by the uniqueness of Haar measure (on the special orthogonal group $SO(n)$). Then, since the covariance matrix $\Sigma$ is symmetric and $R$ is an angular function, we have
$$R \bigl(\Sigma^{1/2} N_1 \times \ldots \times \Sigma^{1/2} N_{d-1} \bigr) = R \bigl(\Sigma^{-1/2}  (N_1 \times \ldots \times  N_{d-1}) \bigr) \stackrel{d}{=} R(\Sigma^{-1/2} U).$$ Combining this fact with \eqref{pm}, \eqref{eq: uniform discrete 1}, and \eqref{eq: uniform discrete 2} yields the main assertion of Proposition~\ref{face_prob asympt}.


{\bf Proof of Theorem~\ref{second}.} A straightforward extension of the path-transform argument in the proof of \eqref{temporal} in
Proposition~\ref{face_prob} gives a little strengthening of \eqref{eq: argmin}: for any non-negative Borel function $g: \R^{d \times(d-1)} \to \R$,
\begin{align*}
& \E \Bigl[ g(S_{i_2} - S_{i_1}, \dots, S_{i_d} - S_{i_{d-1}}) \I (0, S_1, \dots, S_n \in H_\pm(S_{i_1}, \dots, S_{i_d}) ) \Bigr] \\
&=\E \Bigl[ g(S_{i_2-i_1}, S_{i_3-i_1} - S_{i_2-i_1}, \dots, S_{i_d -i_1} - S_{i_{d-1} -i_1}) \I \bigl(S_1, \dots, S_{i_d - i_1} \in H_\pm(0, S_{i_2-i_1}, \dots, S_{i_d-i_1} \bigr) \Bigr. \\
& \qquad \times \Bigl. \I \bigl( \argmin \limits_{0 \le k \le n -(i_d-i_1)} \det[ S_{i_2-i_1}, \dots, S_{i_d-i_1} , S_{i_d-i_1+ k} ] = \{i_1\} \Bigr) \Bigr].\\
\end{align*}
For any fixed tuple $(i_2-i_1, \dots, i_d-i_1)=:(i_1', \dots,i_{d-1}')$, we sum over $i_1=:i$ from $0$ to $n-i'_{d-1}$ to obtain a conditional version of \eqref{temporal}:
\begin{align} \label{temporal cond}
\sum_{i=0}^{n-i_{d-1}'} \E \bigl(\conv(S_i, S_{i+i_1'},\dots,S_{i+ i_{d-1}'} ) \in \mathcal F_n &\bigl| \bigr. S_{i+i_1'} - S_{i},S_{i+i_2'} - S_{i+i_1'},\dots,S_{i+ i_{d-1}'} - S_{i+ i_{d-2}'} \bigr) \notag\\
&= \frac{2}{i_1'(i_2'-i_1') \cdot ... \cdot (i_{d-1}'-i_{d-2}')} \text{ a.s}.
\end{align}
Theorem~\ref{second} then follows immediately by summation over all temporal structures $(i_1', \dots,i_{d-1}')$.

{\bf Computation of the asymptotics.} 1. We claim that for any sequence $a_n$ such that $a_n \sim (\log n)^a n^{-1/2}$ for some $a \ge 0$, it holds that
\be \label{sv asympt}
\sum_{k=1}^n \frac{a_{n-k}}{k} \sim \frac{(\log n)^{a+1}}{\sqrt{n}}.
\ee
In particular, by \eqref{gamma asympt} this implies \eqref{log asympt} if we take $a=0$.

Let us check that the main contribution to the asymptotics in~\eqref{sv asympt} comes from the terms $k=o(n)$. Since the sum of $a_n$ diverges, for any $\varepsilon \in (0,1)$,
$$\sum_{k=\varepsilon n }^n \frac{a_{n-k}}{k}  \le \frac{1}{\varepsilon n} \sum_{k=\varepsilon n }^n a_{n-k} \sim  \frac{1}{\varepsilon n} \sum_{k=1}^{(1-\varepsilon) n } \frac{(\log k)^a}{\sqrt{k}} \le \frac{2 (\log n )^a}{\varepsilon \sqrt{n}}.$$ The last expression is of a smaller order of asymptotics than
$$\sum_{k=1}^{\varepsilon n - 1} \frac{a_{n-k}}{k} \sim \sum_{k=1}^{\varepsilon n-1} \frac{(\log (n-k))^a}{k \sqrt{(n-k)}},$$ since
$$\frac{(\log n)^{a+1}}{\sqrt{n}} \sim  \frac{(\log (1-\varepsilon) n )^a}{\sqrt{n}} \sum_{k=1}^{\varepsilon n-1} \frac{1}{k}  \le \sum_{k=1}^{\varepsilon n-1} \frac{(\log (n-k))^a}{k \sqrt{(n-k)}} \le \frac{(\log n)^a}{\sqrt{(1-\varepsilon) n}} \sum_{k=1}^{\varepsilon n-1} \frac{1}{k} \sim  \frac{(\log n)^{a+1}}{\sqrt{(1-\varepsilon) n}}.$$ These inequalities clearly imply \eqref{sv asympt}.

2. We prove \eqref{E RW 0} by induction in $d$. The base $d=2$ holds by \eqref{log asympt}, which we proved above. Since
\beaa
&& \sum_{1 \leq i_2 <\dots<i_d\leq n}\frac{(2n-2i_d-1)!!}{i_2 \cdot (2n-2i_d)!!}\prod_{k=2}^{d-1}\frac{1}{i_{k+1}-i_k} \\
&=& \sum_{i_2 = 1}^{n-d+2} \frac{1}{i_2} \left[\sum_{1 \leq i_2' <\dots<i_{d-1}' \leq n - i_2} \frac{(2(n-i_2)-2i_{d-1}'-1)!!}{i_2' \cdot (2(n-i_2)-2i_{d-1}')!!}\prod_{k=2}^{d-2}\frac{1}{i_{k+1}'-i_k'} \right ],
\eeaa
\eqref{1225} (or \eqref{pinned}) implies that
\be \label{induction}
\E_{RW}^{(d)} |\mathcal{F}_n'|= \sum_{k=1}^n \frac1k \E_{RW}^{(d-1)} |\mathcal{F}_{n-k}'|,
\ee
where the upper indices show dimension and by definition, $\E_{RW}^{(d)} |\mathcal{F}_n'|:=0$ for $n \le d-1$. It remains to use \eqref{sv asympt} to obtain \eqref{E RW 0}.

3. Arguing as above and using \eqref{pinned bridge} instead of \eqref{pinned}, one can easily show that \eqref{induction} also holds for a random walk bridge of length $n+1$. For any sequence $b_n$ such that $b_n \sim (\log n)^b n^{-1}$ for some $b \ge 0$, one has 
\be \label{sv asympt 2}
\sum_{k=1}^n \frac{b_{n-k}}{k} \sim \frac{(b+2)(\log n)^{b+1}}{(b+1) n }.
\ee
The difference with \eqref{sv asympt} is due to the fact that the main contribution to the asymptotics comes from the indices $k$ that are either $k=o(n)$ or $k=n-o(n)$. The asymptotics for the base $d=2$ is then different, namely
$$\E_{Br}^{(2)} |\mathcal{F}_n'| = \sum_{k=1}^n \frac{2}{k(n-k+1)} \sim \frac{4 \log n}{n},$$
but the rest is analogous and \eqref{E BR 0} follows easily.

4. The assertion \eqref{E RW 0 assymetric} immediately follows from \eqref{E RW 0} once we check that in both summations resulting in these asymptotics, for any slowly varying sequence $c_n$ tending to infinity, the contributions of the indices $1 \le i_2 < i_3 < \ldots < i_d \le n$ with $\min(i_2, i_3 -i_2, \ldots, i_d - i_{d-1}, n-i_d) \le c_n$ are of a smaller order of asymptotics. 

We already saw that the main contribution to the asymptotics of the sum in \eqref{sv asympt} comes from the indices $k=o(n)$. Consequently, the indices with $n - i_d \le c_n$ do not contribute to the asymptotics in \eqref{E RW 0}  and \eqref{E RW 0 assymetric}. On the other hand,
$$\sum_{k=1}^{c_n} \frac{a_{n-k}}{k} \sim \sum_{k=1}^{c_n} \frac{(\log(n-k))^a}{k \sqrt{n-k}} \sim \frac{(\log n)^a}{\sqrt{n}} \sum_{k=1}^{c_n} \frac{1}{k} \sim \log c_n \frac{(\log n)^a}{\sqrt{n}} = o \Bigl( \frac{(\log n)^{a+1}}{\sqrt{n}} \Bigr),$$ hence the indices $k \le c_n$ do not contribute as well to the sum in \eqref{sv asympt}. Consequently, neither do  any of the indices satisfying $\min(i_2, i_3 -i_2, \ldots, i_d - i_{d-1}) \le c_n$.

\section*{Appendix}
\begin{proof}[\bf Proof of Lemma~\ref{lem: half asympt}.]
For any direction $u \in \S^{d-1}$, the one-dimensional random walk $S_k^{(u)}:=\langle S_k, u \rangle, k \ge 1,$ has increments $\langle X_k, u \rangle$ with zero mean and strictly positive variance $\langle \Sigma u, u \rangle$; recall that $\Sigma$ is non-degenerate as follows by assumption~\eqref{cond H}. The random variable $T(u)$, which is the exit time of the random walk $S_k$ from the half-space $H(u)$, coincides with the exit time of the walk $S_k^{(u)}$ from the non-negative half-line.
Then 
$$R(u)  = -\frac{\E \langle S_{T(u)}, u \rangle}{\sqrt{\langle \Sigma u, u \rangle }} = 
-\frac{\E  S_{T(u)}^{(u)}}{\sqrt{\Var(\langle X_1, u\rangle)}}.$$
The last expression admits (Feller~\cite[Section XVIII.5]{Feller}) representation in terms of the so-called Spitzer series:
\begin{equation}
\label{eq: Spitzer}
R(u)  =\frac{1}{\sqrt{2}}\exp \Bigl(\sum_{k=1}^\infty \frac{1}{n} \bigl[\P(S_n^{(u)} > 0) -1/2 \bigr] \Bigr).
\end{equation}
The series it known to converge under the zero mean and finite variance assumption on the increments so $R(u)$ is positive and finite on $\S^{d-1}$.

The convergence in~\eqref{eq: exit halfspace} holds pointwise (cf. \eqref{eq: Spitzer} and Feller~\cite[Section XII.8]{Feller}) for every fixed $H=H(u_H)$. We will show that the standard proof of this statement can be strengthened to obtain the required uniform version. Let us recall this proof. 

For a fixed direction $u \in \S^{d-1}$, we are interested in the asymptotics of the tail probabilities
$$
\P(S_1,\dots, S_n\in H(u))=\P(S_1^{(u)} > 0,\dots, S_n^{(u)}>0) = \P(T(u) > n).
$$
The moment-generating function of $T(u)$ is given by the Spitzer identity
$$1 - \E s^{T(u)} = \exp \Bigl( \sum_{n=1}^\infty \frac{s^n}{n} \P(S_n^{(u)} \ge 0)  \Bigr), \qquad  0 \le s <1,$$
which is valid for any random walk. Since the Spitzer series converges under the zero mean and finite variance assumption on the increments of $S_k^{(u)}$, we have
$$1 - \E s^{T(u)} = \sqrt{1-s}\exp \Bigl( \sum_{n=1}^\infty \frac{s^n}{n} \bigl[\P(S_n^{(u)} \ge 0) -1/2 \bigr] \Bigr).$$ 
Then 
$$\sum_{n=0}^\infty \P(T(u) > n) s^n = \frac{1}{\sqrt{1-s}}\exp \Bigl( \sum_{n=1}^\infty \frac{s^n}{n} \bigl[\P(S_n^{(u)} \ge 0) -1/2 \bigr] \Bigr),$$ which can be verified by summation by parts in the l.h.s. 

Since the Spitzer series converges, by Abel's theorem and \eqref{eq: Spitzer} we have
\begin{equation}
\label{eq: Abel}
\lim_{s \to 1-} \exp \Bigl( \sum_{n=1}^\infty \frac{s^n}{n} \bigl[\P(S_n^{(u)} \ge 0) -1/2 \bigr] \Bigr)= \sqrt{2} R(u).
\end{equation}
Hence
\begin{equation}
\label{eq: MGF equiv}
\sum_{n=0}^\infty \P(T(u) > n) s^n \sim \frac{\sqrt{2} R(u)}{\sqrt{1-s}}, \qquad s \to 1-,
\end{equation}
and the pointwise version of \eqref{eq: exit halfspace} follows by a Tauberian theorem for power series of sequences with monotone differences. In fact, $U(n):= \sum_{k=0}^n \P(T(u) > k)$ has monotone differences $U(n) - U(n+1) = \P(T(u) > n)$. 

Now we explain how to modify the above argument to obtain the uniform asymptotics. The key ingredient is that the Spitzer series converges absolutely\footnote{For our proof, it actually suffices to use uniform convergence rather than the uniform absolute convergence. Indeed, it can be shown using Abel's uniform convergence test that the convergence in~\eqref{eq: Abel} is uniform as required.} uniformly in $u \in \S^{d-1}$. This is true by Lemma~\ref{lem: Uniform Spitzer} below applied to the family of random variables $\frac{\langle X_1, u \rangle}{\sqrt{\Var(\langle X_1, u\rangle)}}, u \in \S^{d-1},$ which is uniformly square integrable by the inequality $\frac{\langle X_1, u \rangle^2}{\Var(\langle X_1, u\rangle)} \le \sigma_1^{-1}\|X_1\|^2_2$, where $\sigma_1$ denotes the smallest eigenvalue of $\Sigma$. 


Since the Spitzer series converges absolutely uniformly in $u$ and it dominates termwise the absolute values of the series in~\eqref{eq: Abel}, the convergence in~\eqref{eq: Abel} is uniform.
Then the equivalence in~\eqref{eq: MGF equiv} is also uniform in $u \in \S^{d-1}$, and by the second assertion of the uniform Tauberian Theorem~\ref{thm: Uniform Tauberian} below, this implies \eqref{eq: exit halfspace}.

Finally, note that each term of the Spitzer series, namely $n^{-1}\bigl[\P(S_n^{(u)} > 0) -1/2 \bigr]$, depends continuously on $u \in \S^{d-1}$. This is readily seen from the continuity of probability measures and the fact that the distribution of $S_n$ does not put mass on hyperplanes due to assumption~\eqref{cond H}. Then $R(u)$ is continuous on $\S^{d-1}$ as a uniform limit of continuous functions. 

\end{proof}

\subsection*{Uniform absolute convergence of the Spitzer series} We present a statement stronger than needed for the use in the current paper. 

\begin{lem}
\label{lem: Uniform Spitzer}
Let $\{Y_\alpha\}_{\alpha \in I},$ where $I$ is some index set, be random variables with zero mean and unit variance. Let  $S_n^{(\alpha)}, n \ge 1,$ be a random walk with increments distributed as $Y_\alpha, \alpha \in I$. If the family $\{ Y_\alpha \}_{\alpha \in I}$ is uniformly square integrable, then the series
$$\sum_{n=1}^\infty \frac{1}{n} \bigl|\P(S_n^{(\alpha)} \ge 0) -1/2 \bigr|$$
converges uniformly in $\alpha \in I$.
\end{lem}

This statement fully rests on the series remainder estimate by Nagaev~\cite{Nagaev2009}.

\begin{proof}
As in~\cite{Nagaev2009}, for any $\alpha \in I$ denote $n_1(\alpha) = \min(k \ge 1: \E Y_\alpha^2 \I_{\{ Y_\alpha^2 < k\}}> 3/4)$ and $n_0(\alpha) = \max(8,n_1(\alpha))$. Putting together Eq.'s (6), (9) and (10) from~\cite{Nagaev2009} that estimate the terms of the main bound Eq. (2) gives that for any $k \ge n_0(\alpha)$,
$$\sum_{n=k}^\infty \frac{1}{n} \bigl|\P(S_n^{(\alpha)} \ge 0) -1/2 \bigr| \le \frac{19}{4 \sqrt{k}} + \frac{3}{\sqrt2} \E Y_\alpha^2 \I_{\{ Y_\alpha^2 \ge k\}} + \frac{2}{\sqrt{k}}\E |Y_\alpha|^3 \I_{\{|Y_\alpha| \le \sqrt{k}\}} + 4 \E |Y_\alpha| \I_{\{|Y_\alpha| > \sqrt{k}\}}.$$
The only difference with Nagaev's estimates is that this inequality is obtained by summation in Eq. (2) over $n \ge k$ rather than $n \ge n_0(\alpha)$ as is~\cite{Nagaev2009}. We also introduced a minor correction to Eq. (9). 

Since $n_0:= \sup_\alpha n_0(\alpha)$ is finite by the uniform square-integrability, the remainder estimate applies to all $\alpha \in I$ if $k$ is large enough. The first term vanishes as $ k \to \infty$ and by the uniform square-integrability, so does the second one uniformly in $\alpha \in I$. For the fourth term, use the Cauchy--Bunyakovsky--Schwarz inequality. For the remaining third term, for any $\varepsilon >0$, we have
$$\frac{1}{\sqrt{k}}\E |Y_\alpha|^3 \I_{\{|Y_\alpha| \le \sqrt{k}\}} \le \varepsilon \E Y_\alpha^2 \I_{\{|Y_\alpha| \le \varepsilon \sqrt{k}\}} + \E Y_\alpha^2 \I_{\{\varepsilon \sqrt{k} \le |Y_\alpha| \le \sqrt{k}\}} \le  \varepsilon + \E Y_\alpha^2 \I_{\{\varepsilon \sqrt{k} \le |Y_\alpha| \}},$$ where the last term again vanishes uniformly.
\end{proof}

\subsection*{A uniform Tauberian theorem} Although Tauberian theory is a very well studied subject and there are many results on the remainder terms in asymptotics, to our surprise we did not find any reference on uniform convergence. The next result is presented in a greater generality than needed for the use in the current paper. 

\begin{theo}[Uniform Tauberian theorem]
\label{thm: Uniform Tauberian}
Let $\{U_\alpha\}_{\alpha \in I}$, where $I$ is some index set, be non-decreasing right-continuous functions on $\R$ with $U_\alpha(0-)=0$ for every $\alpha \in I$, and let $\{L_\alpha \}_{\alpha \in I}$ be slowly varying functions. Assume that for some $\rho \ge 0$ 
$$\hat{U}_\alpha(s):= \int_{0}^\infty e^{-s x} dU_\alpha(x) \sim  s^{-\rho} L_\alpha(1/s), \quad s \to 0+ \quad \text{uniformly in } \alpha \in I.$$
Then
$$U_\alpha(x) \sim \frac{x^{\rho} L_\alpha(x)}{\Gamma(1 + \rho)}, \quad x \to \infty \quad \text{uniformly in } \alpha \in I.$$ 
If in addition, $U_\alpha$ is absolutely continuous with a monotone density $u_\alpha$ and $L_\alpha(x) \equiv c_\alpha$ is a positive constant for every $\alpha \in I$, and $\rho >0$, 
then
\begin{equation}
\label{eq: uniform diff}
u_\alpha(x) \sim \frac{x^{\rho-1} L_\alpha(x)}{\Gamma(\rho)}, \quad x \to \infty \quad \text{uniformly in } \alpha \in I.
\end{equation}
\end{theo}
\begin{rem*}
It is possible to show that~\eqref{eq: uniform diff} holds under the less strenuous (than $L_\alpha(x) \equiv c_\alpha$) assumption of uniform slow variation for $\{L_\alpha \}_{\alpha \in I}$:
$$\lim_{x \to \infty} \sup_{1 \le s \le 2} \Bigl | \frac{L_\alpha(sx)}{L_\alpha(x)} -1 \Bigr| =0 \quad \text{uniformly in } \alpha \in I.$$
\end{rem*}
Our proof fully follows the one of Korevaar's~\cite[Theorem I.15.3]{Korevaar}, which is based on explicit estimates of $U_\alpha(x)$ as opposed to more elegant standard proofs (as in Feller~\cite[Theorem XIII.5.2]{Feller}) relying on the continuity theorem for  Laplace transform. 
\begin{proof}
For any positive integer $m$, 
\begin{equation}
\label{eq: unif equiv}
\hat{U}_\alpha(ks) \sim  k^{-\rho} s^{-\rho} L_\alpha(1/s), \quad s \to 0+ \quad \text{uniformly in } \alpha \in I, k \in \{1, \ldots, m\}.
\end{equation} 
Then, since for any positive integer $k$, one has
$$\int_{0}^\infty e^{-kx} d(x^\rho) = k^{-\rho} \Gamma(1+ \rho),$$
we see from \eqref{eq: unif equiv} that for any polynomial $P(z)=\sum_{k=1}^m a_k z^k$,
\begin{equation}
\label{eq: polynomial equiv}
\int_{0}^\infty P(e^{-sx}) dU_\alpha(x) \sim  \frac{s^{-\rho} L_\alpha(1/s)}{\Gamma(1+\rho)} \int_0^\infty P(e^{-x}) d (x^\rho) , \quad s \to 0+ \quad \text{uniformly in } \alpha \in I.
\end{equation}

As in~\cite[Theorem I.15.3]{Korevaar}, denote $g(z):=\I_{[e^{-1},1]}(z)$ and for any $\varepsilon >0$, consider a polynomial $P_\varepsilon(z)$ approximating the indicator function $g(z)$ on $[0,1]$ such that 
$$P_\varepsilon(z) \ge g(z), z \in [0,1], \quad \text{and} \quad \int_0^1 (P_\varepsilon(z) - g(z))  \rho (-\log z)^{\rho -1} z^{-1} dz \le \varepsilon.$$ The latter condition ensures that
$$\int_0^\infty P_\varepsilon(e^{-x}) d(x^\rho) \le  \int_0^\infty g(e^{-x}) d(x^\rho)  + \varepsilon  = \int_0^1 d (x^\rho) + \varepsilon = 1 + \varepsilon.$$ Finally, since by the choice of $P_\varepsilon$,
$$\int_{0}^\infty P_\varepsilon(e^{-sx}) dU_\alpha(x) \ge \int_{0}^\infty g(e^{-sx}) dU_\alpha(x) = U_\alpha(1/s),$$ from \eqref{eq: polynomial equiv} we see that there exists an $s_\varepsilon>0$ such that
$$U_\alpha(1/s) \le (1 + \varepsilon) \frac{s^{-\rho} L_\alpha(1/s)}{\Gamma(1+\rho)},  \quad \alpha \in I, s \in (0, s_\varepsilon).$$ 

Similarly, we obtain an analogous lower bound. Both inequalities imply the first assertion of the theorem. 

The second assertion~\eqref{eq: uniform diff} that the uniformity is preserved under ``differentiation'' of the asymptotics can be checked by repeating the elementary proof of Lemma~17.1 in~\cite{Korevaar}. We omit the details. The assertion of the remark follows along the same lines.

\end{proof}

\section*{Acknowledgement}
We wish to thank Nick Bingham for discussions and advice. We are grateful to the anonymous referees for their comments and suggestions.

\bibliographystyle{plain}
\bibliography{halfplane_bib}
\end{document}